\documentclass{amsart}
\usepackage{esint}
\usepackage{amsfonts}
\usepackage{amssymb}
\usepackage{xcolor}
\usepackage{url, hyperref}

\definecolor{limegreen}{rgb}{0.196,0.804,0.196}
\definecolor{darkgreen}{rgb}{0.0,0.5,0.0}
\definecolor{darkbluegreen}{rgb}{0,0.3,0.6}
\definecolor{badgerred}{rgb}{0.715,0.004,0.004}
\hypersetup{colorlinks, 
citecolor=badgerred,
filecolor=black,
linkcolor=blue,
urlcolor=darkgreen,
bookmarksopen=false,
pdftex}
\newtheorem{theorem}{Theorem}[section]

\newtheorem{corollary}[theorem]{Corollary}

\newtheorem{definition}[theorem]{Definition}

\newtheorem{lemma}[theorem]{Lemma}

\newtheorem{proposition}[theorem]{Proposition}

\numberwithin{equation}{section}

\begin{document}
\title[Positive solutions and applications]{Positive solutions to
Schr\"odinger equations and geometric applications}
\author{Ovidiu Munteanu, Felix Schulze and Jiaping Wang}

\begin{abstract}
A variant of Li-Tam theory, which associates to each end of a complete
Riemannian manifold a positive solution of a given Schr\"odinger equation on
the manifold, is developed. It is demonstrated that such positive solutions
must be of polynomial growth of fixed order under a suitable scaling
invariant Sobolev inequality. Consequently, a finiteness result for the number of ends
follows. In the case when the Sobolev inequality is of particular type, the finiteness result 
is proven directly. As an application, an estimate on the number of ends for shrinking 
gradient Ricci solitons and submanifolds of Euclidean space is obtained. 
\end{abstract}

\address{Department of Mathematics, University of Connecticut, Storrs, CT
06268, USA}
\email{ovidiu.munteanu@uconn.edu}
\address{Department of Mathematics, University of Warwick, Coventry CV7 4AL,
UK}
\email{felix.schulze@warwick.ac.uk}
\address{School of Mathematics, University of Minnesota, Minneapolis, MN
55455, USA}
\email{jiaping@math.umn.edu}
\maketitle

\section{Introduction}

Recall that a complete manifold $\left( M,g\right) $ is a gradient shrinking
Ricci soliton if there exists a function $f$ on $M$ such that the Ricci
curvature of $M$ and the hessian of $f$ satisfy the equation 
\begin{equation*}
\mathrm{Ric}+\mathrm{Hess}(f)=\frac{1}{2}\,g.
\end{equation*}%
As self-similar solutions to the Ricci flow, gradient shrinking Ricci
solitons arise naturally from singularity analysis of the Ricci flow.
Indeed, according to \cite{Se, N, EMT, CZh}, the blow-ups around a type-I
singularity point always converge to nontrivial gradient shrinking Ricci
solitons. It is thus a central issue in the study of the Ricci flow to
understand and classify gradient shrinking Ricci solitons. While the issue
has been successfully resolved for dimension $2$ and $3$ (see \cite{H, P, N,
NW, CCZ}), it remains open for dimension $4,$ though recent work \cite{MW1,
MW2, CDS} has shed some light on it. Presently, there is very limited
information available concerning general gradient shrinking Ricci solitons
in higher dimensions.

The potential $f$ and the scalar curvature $S$ are related through the
following equation \cite{H} 
\begin{equation}
\left\vert \nabla f\right\vert ^{2}+S=f  \label{fsintr}
\end{equation}%
with $f$ normalized by adding a suitable constant. 
By \cite{Ch}, $S>0$ unless $(M,g)$ is the Euclidean space. Moreover,
according to \cite{CZ, HM}, there exists a point $p\in M$ and constants 
$c_1\left( n\right),$  $c_2\left( n\right)$ depending only on the dimension $n$ of $M$ such that 
\begin{equation}
\frac{1}{4}r^{2}(x)-c_1(n)r(x)-c_2(n)\leq f(x)\leq \frac{1}{4}r^{2}(x)+c_1(n)r(x)+c_2(n)
\label{fintr}
\end{equation}%
for all $x\in M,$ where $r(x)=d(p,x)$ is the distance from $p$ to $x$, and
the volume $\mathrm{V}_{p}(r)$ of the geodesic ball $B_{p}(r)$  
centered at $p$ of radius $r$ satisfies 
\begin{equation}
\mathrm{V}_{p}(r) \leq c(n)\, r^{n}. \label{vintr}
\end{equation}
Perelman's entropy is given by 
\begin{equation}
\mu(g) =\ln \left( \frac{1}{\left( 4\pi \right) ^{\frac{n}{2}}}%
\int_{M}e^{-f}\right).  \label{muintr}
\end{equation}%
Set 
\begin{equation}
\alpha =\limsup_{R\rightarrow \infty } \frac{1}{\mathrm{V}_{p}(R) }
\int_{B_{p}\left( R\right)}\left( S\,r^2\right) ^{\frac{n-1}{2}}.
\label{alintr}
\end{equation}

We have the following result.

\begin{theorem}
\label{T1intr} Let $\left( M,g\right) $ be a gradient shrinking Ricci
soliton with $\alpha<\infty.$ Then the number of ends of $M$ is bounded
from above by $\Gamma( n,\alpha ,\mu (g)),$ a constant depending only on
dimension $n,$ $\mu(g) $ and $\alpha.$ 
\end{theorem}

A gradient shrinking Ricci soliton $M$ is called asymptotically conical if
there exists a closed Riemannian manifold $(\Sigma, g_{\Sigma })$ and
diffeomorphism 
\begin{equation*}
\Phi: (R,\infty )\times \Sigma \rightarrow M\setminus \Omega
\end{equation*}%
such that $\lambda ^{-2}\,\rho _{\lambda }^{\ast }\,\Phi ^{\ast }\,g$
converges in $C_{loc}^{\infty }$ as $\lambda \rightarrow \infty $ to the
cone metric $dr^{2}+r^{2}\,g_{\Sigma }$ on $[R,\infty )\times \Sigma,$ where 
$\Omega$ is a compact smooth domain of $M.$ Clearly, an asymptotically
conical shrinking Ricci soliton must satisfy $\alpha<\infty.$ 

Recall that an end of a complete manifold $M$ with respect to a compact
smooth domain $\Omega\subset M$ is simply an unbounded component of 
$M\setminus \Omega.$ The number of ends $e(M)$ of $M$ is the maximal number
obtained over all such $\Omega.$ The novelty of Theorem \ref{T1intr} is that
only the scalar curvature integral information at infinity is needed.
Another feature is that the exponent of $S$ in the definition of $\alpha$ is 
$\frac{n-1}{2},$ not the commonly seen $\frac{n}{2}$ in analysis. We
emphasize that the estimate here is explicit. That $M$ has 
finitely many ends follows readily by assuming the scalar curvature of $M$ is bounded. Indeed, as
observed in \cite{FMZ}, (\ref{fsintr}) and (\ref{fintr}) imply that $%
\left\vert \nabla f\right\vert\geq 1$ outside a compact subset of $M$ and hence $M$
must have finite topological type. We mention here that in \cite{MW} it was 
shown that any complete shrinking K\"{a}hler Ricci soliton must have one end. 
The proof uses Li-Tam's theory and a fact special to the K\"ahler situation that the gradient vector 
$\nabla f$ is real holomorphic.

For shrinking gradient Ricci solitons of dimension $n\geq 3,$ by Li-Wang \cite{LW},
the following Sobolev inequality holds.
\begin{equation*}
\left( \int_{M}\phi ^{\frac{2n}{n-2}}\right) ^{\frac{n-2}{n}}\leq 
C(n)\,e^{-\frac{2\mu \left( g\right) }{n}}\int_{M}\left( \left\vert \nabla \phi
\right\vert ^{2}+S\phi ^{2}\right) 
\end{equation*}%
for $\phi \in C_{0}^{\infty }(M).$ So Theorem \ref{T1intr} is a consequence
of the following general result.

\begin{theorem}\label{T2intr}
Let $\left( M,g\right) $ be a complete Riemannian manifold of dimension $n\geq 3$
satisfying the Sobolev inequality 
\begin{equation*}
\left( \int_{M}\phi ^{\frac{2n}{n-2}}\right) ^{\frac{n-2}{n}}\leq A
\int_{M}\left(\left\vert \nabla \phi \right\vert ^{2}+\sigma \phi
^{2}\right)
\end{equation*}%
for any $\phi \in C_{0}^{\infty }\left( M\right),$ where $A>0$ is a
constant and $\sigma \geq 0$ a continuous function. Suppose

\begin{equation*}
\alpha =\limsup_{R\rightarrow \infty }\frac{1}{\mathrm{V}_{p}\left(R\right) }\int_{B_p\left(R\right) }\left( r^{2}\sigma \right) ^{\frac{n-1}{2}}<\infty 
\end{equation*}
and
\begin{equation*}
V_{\infty }=\limsup_{R\rightarrow \infty }\frac{\mathrm{V}_p\left(R\right) }{R^{n}}<\infty . 
\end{equation*}
Then the number of ends of $M$ is bounded above by a constant $\Gamma$
depending only on $n,$ $A,$ $\alpha$ and $V_{\infty }.$
\end{theorem}

The well known Michael-Simon inequality \cite{A, MS} for submanifolds in the
Euclidean space $\mathbb{R}^N$ states that
\begin{equation}
\left( \int_{M}\left\vert \phi \right\vert^{\frac{n}{n-1}}\right) ^{\frac{n-1}{n}}\leq
C(n)\int_{M}\left( \left\vert \nabla \phi \right\vert +\left\vert
H\right\vert \left\vert \phi \right\vert\right)   \label{MSintr}
\end{equation}%
for any $\phi \in C_{0}^{\infty }(M),$ where $H$ is the mean curvature vector of $M.$
In fact, this inequality holds for submanifolds in Cartan-Hadamard
manifolds as well \cite{HS}. These inequalities are particularly useful in
studying minimal submanifolds. We refer to \cite{Ca, Ni, CSZ, PRS} and the references 
therein for some of the results. It is easy to see that 
\begin{equation*}
\left( \int_{M}\phi ^{\frac{2n}{n-2}}\right) ^{\frac{n-2}{n}}\leq
C(n)\int_{M}\left( \left\vert \nabla \phi \right\vert^2 +\left\vert
H\right\vert^2 \phi ^2\right)
\end{equation*}%
holds for $n\geq 3.$ As a corollary of Theorem \ref{T2intr}, we have the following result.

\begin{corollary}\label{T3intr}
Let $M^n$ be a complete submanifold of $\mathbb{R}^N$ with $n\geq 2.$
Suppose

\begin{equation*}
\alpha =\limsup_{R\rightarrow \infty }\frac{1}{\mathrm{V}_{p}\left(R\right) }\int_{B_p\left(R\right) } \left(r\,\left\vert H\right\vert\right)^{n-1}<\infty 
\end{equation*}
and
\begin{equation*}
V_{\infty }=\limsup_{R\rightarrow \infty }\frac{\mathrm{V}_p\left(R\right) }{R^{n}}<\infty. 
\end{equation*}
Then the number of ends of $M$ is bounded above by a constant $\Gamma$
depending only on the dimension $n,$ $\alpha$ and $V_{\infty }.$
\end{corollary}

Strictly speaking, for the case of dimension $n=2,$ the conclusion does not follow directly from Theorem \ref{T2intr}. Rather,
it follows by a slight modification of its proof. 
Our proof of Theorem \ref{T2intr} is very much motivated by the work of Topping \cite{T1,T2}, 
where the diameter of a compact manifold $M$ satisfying the Sobolev inequality is estimated in terms of the constant $A$
together with the integral $\int_M \sigma^{\frac{n-1}{2}}.$ The argument there is adapted to show that for each large $R,$ 
the volume of $E\cap B_p(R)$ satisfies $\mathrm{V}\left(E\cap B_p(R)\right)\geq c\,R^n$ for some constant $c$ for
at least one half of the ends $E$ of $M.$ Note that for different $R$ the choice of such set of ends $E$ may be different. Nonetheless,
the desired estimate
of the number of ends follows as the total volume of the ball $B_p(R)$ is at most of $2V_{\infty}\,R^n.$
We emphasize that the argument strongly depends on the fact that the Sobolev exponent is of $\frac{n}{n-2}$ 
with $n$ being the dimension of the manifold.
For a Sobolev inequality with general exponent $\mu>1$ of the form
\begin{equation*}
\left( \int_M \phi ^{2\mu }\right) ^{\frac{1}{\mu }}\leq A\,\int_{M }\left( \left\vert \nabla \phi \right\vert
^{2}+\sigma \phi ^{2}\right) 
\end{equation*}%
for $\phi \in C_{0}^{\infty }\left( M\right),$
we instead develop a different approach of using positive solutions to a Schr\"odinger equation 
to estimate the number of ends.

More specifically, the approach relies on a variant of Li-Tam theory. In 
\cite{LT}, to each end $E$ of $M,$ they associate a harmonic function $f_E$
on $M.$ The resulting harmonic functions are linearly independent. So the
question of bounding the number of ends $e(M)$ is reduced to estimating the
dimension of the space spanned by those functions. The theory was
successfully applied to show that $e(M)$ is necessarily finite when the
Ricci curvature of $M$ is nonnegative outside a compact set. We shall refer
to \cite{Li} for more applications of this theory. Here, we develop a
variant of their theory by considering instead the Schr\"odinger operator
\[
L=\Delta-\sigma
\] 
with $\sigma$ being a nonnegative but not identically zero
smooth function on $M.$

\begin{theorem}
\label{Eintr} Let $\left( M,g\right) $ be a complete manifold and $%
E_{1},E_{2},\cdots ,E_{l}$ the ends of $M$ with respect to a geodesic ball $%
B_{p}(r_{0})$ of $M$ with $l\geq 2.$ Then for each end $E_{i},$ there exists
a positive solution $u_{i}$ to the equation $\Delta u_{i}=\sigma u_{i}$ on $%
M $ satisfying $0<u_{i}\leq 1$ on $M\backslash E_{i}$ and 
\begin{equation*}
\sup_{M}u_{i}=\limsup_{x\rightarrow E_{i}(\infty )}u_{i}(x)>1.
\end{equation*}%
Moreover, the functions $u_{1},\cdots ,u_{l}$ are linearly independent.
\end{theorem}

One nice feature here is that all the functions $u_i$ are positive, while in
the case of harmonic functions $f_E$ is positive if and only if $M$ is
nonparabolic, that is, it admits a positive Green's function. With this
result in hand, we set out to bound the dimension of the space $\mathcal{F}$
spanned by the functions $u_1,\cdots,u_l.$ The work of \cite{CM1, CM2, L} on
the dimension of spaces of harmonic functions with polynomial growth inspires
us to consider the mean value property for positive subsolutions to $L.$
More precisely, assume that $M$ admits a proper Lipschitz function $\rho >0$
satisfying 
\begin{equation}
\frac{1}{2}\leq \left\vert \nabla \rho \right\vert \leq 1\text{ \ and \ }%
\Delta \rho \leq \frac{m}{\rho }\, ,  \label{L}
\end{equation}%
in the weak sense for $\rho \geq R_{0},$ a sufficiently large constant and
some constant $m>0.$

 Denote the sublevel and level sets of $\rho$ by 
\begin{eqnarray*}
D(r) &=&\left\{ x\in M:\rho \left( x\right) <r\right\} \\
\Sigma( r) &=&\left\{ x\in M:\rho \left( x\right) =r\right\}.
\end{eqnarray*}
To simplify notation, we let $\mathrm{V}(r)=\mathrm{Vol}(D(r))$ and  $\mathrm{A}(r)=\mathrm{Area}(\Sigma(r))$.

\begin{definition}
A manifold $\left( M,g\right) $ has the mean value property $\left( \mathcal{%
M}\right) $ if there exist constants $A_{0}>0$ and $\nu >1$ such that for
any $0<\theta \leq 1$ and $R\geq 4R_{0},$ 
\begin{equation}
\sup_{\Sigma \left( R\right) }u\leq \frac{A_{0}}{\theta ^{2\nu }}\frac{1}{%
\mathrm{V}((1+\theta )R)}\int_{D\left( \left( 1+\theta \right) R\right)
\backslash D\left( R_{0}\right) }u  \label{meanintr}
\end{equation}%
holds true for any function $u>0$ satisfying $\Delta u\geq \sigma u$ on $%
D(2R)\backslash D(R_{0}).$
\end{definition}

With this definition at hand, we can now state our main estimate on positive
solutions to the Schr\"odinger equation $L u=0$. For $q\geq 1,$ define the
quantity 
\begin{equation}
\alpha =\limsup_{R\rightarrow \infty }\left( R^{2q}\fint_{\Sigma \left(
R\right) }\sigma ^{q}\right) ^{\frac{1}{q}},  \label{Lq}
\end{equation}
where 
\begin{equation*}
\fint_{\Sigma \left(R\right) }\sigma ^{q}= \frac{1}{\mathrm{A}(R) } \int_{\Sigma \left( R\right) }\sigma
^{q}.
\end{equation*}%
\smallskip

\begin{theorem}
\label{Growthintr}Assume that $\left( M,g\right) $ admits a proper function $%
\rho $ satisfying (\ref{L}) and has the mean value property $\left( \mathcal{%
M}\right) .$ For a polynomially growing positive solution $u$ of $\Delta
u=\sigma u$ on $M\backslash D\left( R_{0}\right) ,$ if $\alpha <\infty $ for
some $q>\nu -\frac{1}{2},$ then there exists a constant $\Gamma \left(
m,A_{0},\nu ,\alpha \right) >0$ such that

\begin{equation*}
u\leq \Lambda \left( \rho ^{\Gamma}+1\right) \text{ \ on } M\backslash
D\left( R_{0}\right),
\end{equation*}%
where $\Lambda>0$ is a constant depending on $u.$ In the critical case $%
q=\nu -\frac{1}{2},$ the same conclusion holds true with $\Gamma=\Gamma(m,
A_0, \nu)$ provided that $\alpha\leq \alpha _{0}( m, A_{0}, \nu),$ a
sufficiently small positive constant.
\end{theorem}

This result is reminiscent of Agmon type estimates in \cite{Ag, LW1, LW2},
where a positive subsolution $u$ to $L$ is shown to decay at a certain rate
if it does not grow too fast, provided that a Poincar\'e type inequality
holds on $M.$ Whether a positive solution $u$ to $L u=0$, under the
assumptions in Theorem \ref{Growthintr}, is automatically of polynomial
growth is unclear at this point. But we do confirm this is the case under a
pointwise assumption on $\sigma >0$ that 
\begin{equation}
\sup_{M}\left( \rho ^{2}\sigma \right) <\infty.  \label{sigmaintr}
\end{equation}

If we let 
\begin{equation*}
L^d(M)=\left\{v: \Delta v=\sigma v, \left\vert v\right\vert \leq c\,\rho ^{d}%
\text{ \ on } M \right\},
\end{equation*}
the space of polynomial growth solutions of degree at most $d,$ then an
argument verbatim to \cite{L} immediately gives the following estimate of
the dimension.

\begin{lemma}
\label{Dimintr} Assume that $\left( M,g\right)$ admits a proper function $%
\rho$ satisfying (\ref{L}) and has mean value property $\left( \mathcal{M}%
\right).$ Then $\dim L^d(M)\leq \Gamma(m, A_0, \nu, d).$
\end{lemma}

Summarizing, we have the following conclusion, where $\mathcal{P}$ is the
space spanned by all positive solutions to the equation $\Delta u=\sigma u$
on $M.$

\begin{theorem}
\label{dimpreintr} Assume that $\left( M,g\right)$ admits a proper function $%
\rho$ satisfying (\ref{L}) and has mean value property $\left( \mathcal{M}%
\right).$ Suppose that $\sigma$ decays quadratically. Then $\dim \mathcal{P}%
\leq \Gamma( m, A_{0},\nu, \alpha)$ provided that $\alpha<\infty$ for some $%
q>\nu -\frac{1}{2}.$ In the critical case $q=\nu -\frac{1}{2},$ the same
conclusion holds for some $\Gamma( m, A_{0},\nu)$ when $\alpha\leq \alpha
_{0}( m, A_{0}, \nu),$ a sufficiently small positive constant. Consequently,
the number of ends $e(M)$ of $M$ satisfies the same estimate as well.
\end{theorem}

It is well known that the mean value property $\left( \mathcal{M}\right)$ is
implied by the following scaling invariant Sobolev inequality via a Moser
iteration argument with the number $\nu$ determined by the Sobolev exponent $%
\mu$ through the equation 
\begin{equation*}
\frac{1}{\mu}+\frac{1}{\nu}=1.
\end{equation*}

\begin{definition}
$\left( M,g\right) $ is said to satisfy the Sobolev inequality $(\mathcal{S}%
) $ if there exist constants $\mu>1$ and $A>0$ such that 
\begin{equation}
\left( \fint_{D\left( R\right) }\phi ^{2\mu }\right) ^{\frac{1}{\mu }}\leq
AR^{2}\fint_{D\left( R\right) }\left( \left\vert \nabla \phi \right\vert
^{2}+\sigma \phi ^{2}\right)  \label{Sobintr}
\end{equation}%
for $\phi \in C_{0}^{\infty }\left( D\left( R\right) \right)$ and $R\geq
R_{0}.$
\end{definition}
We have denoted with 
\[
\fint_{D(R)}u=\frac{1}{\mathrm{V}(R)}\int_{D(R)}u
\]
for any integrable function $u$ on $D(R)$. 
Consequently, Theorem \ref{dimpreintr} continues to hold if one replaces the
mean value property $\left( \mathcal{M}\right) $ by the Sobolev inequality $(%
\mathcal{S}).$  

We also establish a version of Theorem \ref{Growthintr} localized to an end.

For an end $E$ of $M,$ define 
\begin{equation*}
\alpha _{E}=\limsup_{R\rightarrow \infty }\left( \frac{R^{2q}}{\mathrm{A}%
\left( R\right) }\int_{\partial E\left( R\right) }\sigma ^{q}\right) ^{\frac{%
1}{q}},
\end{equation*}%
where $E(R)=E\cap D(R)$ and $\partial E(R)=E\cap \Sigma(R).$

\begin{proposition}
\label{BFLocintr} Assume that $\left( M,g\right)$ admits a proper function $%
\rho$ satisfying (\ref{L}) and that the Sobolev inequality $(\mathcal{S})$
holds. Suppose that $\sigma$ decays quadratically along $E.$ Then there
exists $\Gamma \left( m, A, \mu, \alpha_E\right)>0$ such that 
\begin{equation*}
u\leq \Lambda \left( \rho ^{\Gamma}+1\right) \text{ \ on } E
\end{equation*}%
for positive solutions $u$ to $\Delta u=\sigma u$ on $E,$ where $\Lambda>0$
is a constant depending on $u$, provided that $\alpha_E<\infty$ for some $%
q>\nu -\frac{1}{2}$. In the case $q=\nu -\frac{1}{2},$ the same conclusion
holds for some $\Gamma( m, A,\mu )>0$ when $\alpha_E\leq \alpha _{0}( m, A,
\mu)$, a sufficiently small positive constant.
\end{proposition}

Corresponding to an end $E,$ let $u_{E}$ be the positive solution of $\Delta
u_E=\sigma u_{E}$ on $M$ constructed in Theorem \ref{Eintr}. Then $%
0<u_{E}\leq 1$ on $M\setminus E.$ Proposition \ref{BFLocintr} implies that
such $u_E$ must be of polynomial growth on $M$ with the given growth order.
With this in hand and in view of Lemma \ref{Dimintr}, for the case of
critical $q=\nu -\frac{1}{2},$ one concludes that the number of ends with
small $\alpha_E$ is bounded. For an asymptotically conical gradient
shrinking Ricci soliton $M,$ it is not difficult to show that at least one half
of the ends have small $\alpha_E$ if the total number of ends
is large. Obviously, Theorem \ref{T1intr} follows, at least for asymptotically conical 
shrinking Ricci solitons, from these facts as well.

Sobolev inequalities are prevalent in geometry. Other than the
aforementioned ones for gradient shrinking Ricci solitons and submanifolds in the Euclidean
spaces, for manifolds with Ricci curvature bounded from below by a constant $-K,$ $K\geq 0,$ 
according to \cite{SC1}, the Sobolev inequality (\ref{Sobintr})
holds on any geodesic ball $B_{p}(R) $ with constant $A=e^{c\left(n\right) \left( 1+\sqrt{K}R\right) }$ 
and $\sigma =\frac{1}{R^{2}}.$ Finally, for a locally conformally flat manifold $M,$ by \cite{SY}, a
suitable cover of $M$ can be mapped conformally into $\mathbb{S}^{n}$ and
satisfies a similar Sobolev inequality of gradient shrinking Ricci solitons.

For a comprehensive study of Sobolev inequalities on manifolds and their
applications, we refer to \cite{He, SC}.

The paper is organized as follows. In Section \ref{ToppingEnds}, we present the proof of Theorem \ref{T2intr}
and derive some of its consequences.
In Section \ref{Ends} we focus on the
proof of Theorem \ref{Eintr}. We then turn to estimates of positive
solutions to $\Delta u=\sigma u$ in Section \ref{GR} and prove Theorem \ref%
{Growthintr}. The dimension estimate given in Lemma \ref{Dimintr} is proved
in Section \ref{ED}. Section \ref{Sob} is devoted to proving the fact that
the mean value property $\left( \mathcal{M}\right)$ follows from the Sobolev
inequality $(\mathcal{S})$. 

\section{Sobolev inequality and ends}\label{ToppingEnds}

In this section, we prove Theorem \ref{T2intr} following 
the ideas in \cite{T1,T2}.  To include the case $n=2,$ we consider more generally
complete noncompact Riemannian manifolds $\left( M,g\right) $ satisfying the Sobolev inequality 

\begin{equation}
\left( \int_{M}\vert\phi\vert ^{\frac{q\,n}{n-q}}\right) ^{\frac{n-q}{n}}\leq A
\int_{M}\left(\left\vert \nabla \phi \right\vert ^{q}+\sigma \vert\phi\vert
^{q}\right)  \label{l1}
\end{equation}%
for some $q$ with $1\leq q\leq n-1$ and any $\phi \in C_{0}^{\infty }\left( M\right),$ where $A>0$ is a
constant and $\sigma \geq 0$ a continuous function.
Define 
\begin{equation}
\alpha =\limsup_{R\rightarrow \infty } \frac{1}{\mathrm{V}_p\left(R\right) }
\int_{B_p\left( R\right) }\left( r^{q}\sigma \right) ^{\frac{n-1}{q}} \label{l2}
\end{equation}%
and
\begin{equation}
V_{\infty }=\limsup_{R\rightarrow \infty }\frac{\mathrm{V}_p\left(R\right) }{R^{n}}, \label{l3}
\end{equation}
where $p\in M$ is a fixed point, $r\left( x\right) =d\left( p, x\right) $ is
the distance function to $p,$ and $\mathrm{V}_p\left(R\right) =\mathrm{Vol}\left( B_p\left(R\right) \right),$ 
the volume of the geodesic ball $B_p(R)$ centered at $p$ of radius $R.$

We restate Theorem \ref{T2intr} below under this more general Sobolev inequality.

\begin{theorem}
Let $\left( M,g\right) $ be an $n$ dimensional complete Riemannian manifold satisfying the Sobolev inequality (\ref{l1}). If both
$\alpha$ of (\ref{l2}) and $V_{\infty}$ of (\ref{l3}) are finite, then the number of ends of $M$ is bounded from above by a constant $\Gamma$
depending only on $n,\,A,\,\alpha $ and $V_{\infty }.$
\end{theorem}

\begin{proof}
For an end $E$ of $M$ we denote $E(R) =B_p(R) \cap E.$
Assume that $M$ has at least $k$ ends with $k>1$ large, to be
specified later. We may take $R>0$ large enough such that 

\begin{equation*}
B_p(2R) \backslash B_p(R) =\cup_{i=1}^{k}E_{i}( 2R) \backslash E_{i}( R).
\end{equation*}%
Moreover, we have from (\ref{l3}) that 
\begin{equation}
\frac{\mathrm{V}_p(t) }{t^{n}}\leq 2\,V_{\infty }  \label{l3.1}
\end{equation}%
for all $t\geq R$.
Similarly, by (\ref{l2}) we have, 

\begin{equation*}
\sum_{i=1}^{k}\int_{E_{i}( 3R) \backslash E_{i}( R)
}( r^{q}\sigma) ^{\frac{n-1}{q}}\leq 2\alpha\,\mathrm{V}_p(3R).
\end{equation*}%
This implies that
 \begin{equation}
\sum_{i=1}^{k}\int_{E_{i}( 3R) \backslash E_{i}( R)
}\sigma ^{\frac{n-1}{q}}\leq C_{0}\,\frac{\mathrm{V}_p(3R) }{R^{n-1}}.  \label{l4}
\end{equation}%
Here and below constants $C_0,\; C_1,...$ depend only on $n,\,A,\,\alpha $ and $V_{\infty }.$ 

We may assume that the ends $E_{1},\cdots,E_{k}$ are labeled so that 
\begin{equation*}
\left\{ \int_{E_{i}\left( 3R\right) \backslash E_{i}\left( R\right) }\sigma
^{\frac{n-1}{q}}\right\} _{i=1,\cdots,k}
\end{equation*}%
is an increasing sequence. Then (\ref{l4}) implies that 
\begin{equation}
\int_{E_{i}\left( 3R\right) \backslash E_{i}\left( R\right) }\sigma ^{\frac{%
n-1}{q}}\leq \frac{2C_{0}}{k}\frac{\mathrm{V}_p(3R) }{R^{n-1}}
\label{l5}
\end{equation}%
for all $i=1,2,\cdots,\left[\frac{k}{2}\right]$.

For $i\in \left\{ 1,2,\cdots,\left[ \frac{k}{2}\right] \right\},$ pick
 \begin{equation}
z_{i}\in \partial E_{i}( 2R).  \label{l5.1}
\end{equation}%
By relabeling $E_{1},..,E_{\left[ \frac{k}{2}\right] }$ if necessary, we may assume that 
\begin{equation*}
\left\{ \mathrm{V}_{z_i}(R) \right\} _{i=1,\cdots,\left[ \frac{k}{2}\right] }
\end{equation*}%
is increasing.

Assume by contradiction that
\begin{equation}
\mathrm{V}_{z_1}(R) \geq \frac{C_{1}}{k}\,R^{n},  \label{l6}
\end{equation}%
where $C_{1}=3^{n+2}\,V_{\infty}.$ Since 
\begin{equation*}
B_{z_i}(R) \subset E_{i}( 3R) \backslash E_{i}(R)
\end{equation*}%
and $\left\{ B_{z_i}(R) \right\} _{i=1}^{\left[ \frac{k}{2}\right]}$ are disjoint in 
$B_p(3R),$ it follows from (\ref{l6}) that
\begin{equation*}
\mathrm{V}_p(3R)  \geq \sum_{i=1}^{\left[ \frac{k}{2}\right] }%
\mathrm{V}_{z_i}(R)  
\geq \left[ \frac{k}{2}\right] \frac{C_{1}}{k}\,R^{n}
\geq \frac{C_{1}}{3}\,R^{n}
= 3 V_\infty (3R)^n
\end{equation*}%
as $C_{1}=3^{n+2}V_{\infty }.$ This contradicts (\ref{l3.1}). In
conclusion, (\ref{l6}) does not hold and
\begin{equation*}
\mathrm{V}_{z_1}(R) <\frac{C_{1}}{k}\,R^{n}.
\end{equation*}%
For convenience, from now on we simply write $E=E_{1}$ and $z=z_{1}$.
Hence, we have $z\in \partial E( 2R) $ and 
\begin{equation}
\mathrm{V}_z(R) <\frac{C_{1}}{k}\,R^{n}.  \label{l7}
\end{equation}%
By (\ref{l5}) we also have 
\begin{equation}
\int_{E\left( 3R\right) \backslash E\left( R\right) }\sigma ^{\frac{n-1}{q}%
}\leq \frac{C_{2}}{k}\,\frac{\mathrm{V}_p(3R) }{R^{n-1}}.
\label{l8}
\end{equation}
Let $\gamma( t) $ be a minimizing geodesic from $p$ to $z$
with $0\leq t\leq 2R$. For $t\in \left[ \frac{4}{3}\,R,\frac{5}{3}\,R\right]$ and $x=\gamma( t) $, since
\begin{equation*}
d( x, z) \leq \frac{2}{3}\,R,
\end{equation*}%
the triangle inequality implies 
\begin{equation}
B_x\left(\frac{R}{3}\right) \subset B_z(R).  \label{l8.0}
\end{equation}%
Consequently, (\ref{l7}) yields 
\begin{equation}
\mathrm{V}_x\left(\frac{R}{3}\right) <\frac{C_{1}}{k}\,R^{n} \label{l8.1}
\end{equation}%
for all $x=\gamma( t) $ with $t\in \left[ \frac{4}{3}\,R,\frac{5}{3}\,R\right]$.

Assume by contradiction that 
\begin{equation}
\int_{B_x\left(r\right) }\sigma \leq \delta \,r^{\frac{q}{n-1}}\,\left(\mathrm{V}_x\left(r\right)\right) ^{\frac{n-q-1}{n-1}}  \label{l9}
\end{equation}%
for all $0<r<\frac{R}{3},$ where $\delta >0$ is small constant to be set later.

For $0<r<\frac{R}{3}$ fixed, we apply the Sobolev inequality to cut-off function $\phi $
with support in $B_x( r) $ such that $\phi =1$ on 
$B_x\left(\frac{r}{2}\right) $ and $\left\vert \nabla \phi \right\vert \leq \frac{2}{r}.$
Then (\ref{l1}) implies that 
\begin{equation}
\left(\mathrm{V}_x\left(\frac{r}{2}\right)\right) ^{\frac{n-q}{n}}\leq A\left( \frac{2^q}{r^{q}}\,\mathrm{V}_x\left(r\right) +\int_{B_x\left(r\right) }\sigma \right). \label{I10}
\end{equation}%
Using (\ref{l9}) we obtain that 
\begin{equation}
\left(\mathrm{V}_x\left(\frac{r}{2}\right) \right)^{\frac{n-q}{n}}\leq A\left( \frac{2^q}{r^{q}}\,\mathrm{V}_x(r) +\delta \,
r^{\frac{q}{n-1}}\,\left(\mathrm{V}_x(r)\right)^{\frac{n-q-1}{n-1}}\right)   \label{l11}
\end{equation}%
for any $0<r<\frac{R}{3}.$ Let us assume there exists $0<r<\frac{R}{3}$ so
that 
\begin{equation}
\mathrm{V}_x\left(r\right) \leq \delta^{\frac{n-1}{q}}\,r^{n}.  \label{l12}
\end{equation}%
Then by (\ref{l11}) we have 
\begin{equation*}
\left(\mathrm{V}_x\left(\frac{r}{2}\right) \right)^{\frac{n-q}{n}}\leq A(2^q+1)\,\delta ^{\frac{n-1}{q}}\,r^{n-q}.
\end{equation*}%
Hence,
\begin{equation}
\mathrm{V}_x\left(\frac{r}{2}\right) \leq 2^n\,\left( A(2^{q}+1)\right) ^{\frac{n}{n-q}}\delta ^{\frac{n-1}{n-q}}\,
\delta ^{\frac{n-1}{q}}\,\left( \frac{r}{2}\right) ^{n}.  \label{l13}
\end{equation}%
We now choose $\delta $ to be small enough so that 
\begin{equation*}
2^n\,\left( A(2^{q}+1)\right) ^{\frac{n}{n-q}}\delta ^{\frac{n-1}{n-q}}<1.
\end{equation*}%
Then (\ref{l13}) implies 
\begin{equation}
\mathrm{V}_x\left(\frac{r}{2}\right) \leq \delta ^{\frac{n-1}{q}}\,\left( \frac{r}{2}\right) ^{n}.  \label{l14}
\end{equation}%
In conclusion, assuming that \eqref{l9} holds for any $0<r<\frac{R}{3},$ we have shown that (\ref{l12}) implies (\ref{l14}).

By assuming $k$ to be sufficiently large such that $\frac{3^{n}C_1}{k}\leq \delta ^{\frac{n-1}{q}},$ (\ref{l8.1}) says that 
\begin{equation*}
\mathrm{V}_x\left(\frac{R}{3}\right) \leq \delta ^{\frac{n-1}{q}}\,\left( \frac{R}{3}\right) ^{n},
\end{equation*}%
that is, (\ref{l12}) holds for $r=\frac{R}{3}.$ Applying (\ref{l12}) and (\ref{l14}) inductively, 
we conclude that 
\begin{equation*}
\mathrm{V}_x\left(\frac{R}{3\cdot 2^{m}}\right) \leq \delta ^{\frac{n-1}{q}}\,\left( \frac{R}{3\cdot 2^{m}}\right) ^{n}
\end{equation*}%
for all $m\geq 0.$ Letting $m\rightarrow \infty $ we reach a contradiction by further arranging $\delta$ to be sufficiently small such that $\delta ^{\frac{n-1}{q}}<\omega_n,$
the volume of the unit ball in the Euclidean space $\mathbb{R}^n.$

The contradiction implies that (\ref{l9}) does not hold. Therefore, for any $x=\gamma \left( t\right),$ 
$t\in \left[ \frac{4}{3}\,R,\frac{5}{3}\,R\right],$ there exists $0<r_{x}<\frac{R}{3}$ such that 
\begin{equation}
\int_{B_x\left(r_{x}\right) }\sigma >\delta \left( r_{x}\right) ^{\frac{q}{n-1}}\,
\left(\mathrm{V}_x(r_{x}) \right)^{\frac{n-q-1}{n-1}}.  \label{l16}
\end{equation}%
By the H\"{o}lder inequality we have 
\begin{equation*}
\int_{B_x\left( r_{x}\right) }\sigma \leq \left( \int_{B_x\left(r_{x}\right) }\sigma ^{\frac{n-1}{q}}\right) ^{\frac{q}{n-1}}\,
\left(\mathrm{V}_x\left(r_{x}\right) \right)^{\frac{n-q-1}{n-1}}.
\end{equation*}%
Thus, by (\ref{l16}) we get 

\begin{equation}
\int_{B_x\left(r_{x}\right) }\sigma ^{\frac{n-1}{q}}\geq \frac{1}{C_{3}}\,r_{x}  \label{l17}
\end{equation}%
for any $x=\gamma \left( t\right)$ and $t\in \left[ \frac{4}{3}\,R,\frac{5}{3}\,R\right].$

By a covering argument as in \cite{T1, T2}, we may choose at most countably
many disjoint balls $\left\{ B_{x_m}(r_{x_{m}}) \right\} _{m\geq 1}$ with
$x_{m}=\gamma \left( t_{m}\right),$ $t_{m}\in \left[\frac{4}{3}\,R,\frac{5}{3}\,R\right],$ 
each satisfying (\ref{l17}). Moreover, these balls cover at least one third
of the geodesic $\gamma \left( \left[ \frac{4}{3}\,R,\frac{5}{3}\,R\right] \right).$
Therefore, 
\begin{equation*}
\sum_{m\geq 1}r_{x_{m}} \geq \frac{1}{3}\left( \frac{5}{3}R-\frac{4}{3}R\right) 
=\frac{1}{9}R.
\end{equation*}%
Together with (\ref{l17}) we have 
\begin{equation*}
\frac{1}{9}R \leq \sum_{m\geq 1}r_{x_{m}} 
\leq C_{3}\sum_{m\geq 1}\int_{B_{x_m}\left(r_{x_{m}}\right) }\sigma ^{\frac{n-1}{q}} 
\leq C_{3}\int_{B_z\left(R\right) }\sigma ^{\frac{n-1}{q}},
\end{equation*}%
where for the last inequality we have used (\ref{l8.0}) and that the balls $\left\{ B_{x_m}(r_{x_{m}}) \right\} _{m\geq 1}$ are disjoint in $B_{z}(R)$.

Combining this with (\ref{l8}) and (\ref{l5.1}) we conclude that
\begin{equation*}
\frac{1}{9C_{3}}R \leq \int_{E\left( 3R\right) \backslash E\left( R\right)
}\sigma ^{\frac{n-1}{q}} 
\leq \frac{C_{2}}{k}\frac{\mathrm{V}_p\left(3R\right) }{R^{n-1}}.
\end{equation*}%
In other words,
\begin{equation*}
\mathrm{V}_p(3R) \geq \frac{k}{C_{4}}\,R^{n},
\end{equation*}%
which contradicts (\ref{l3.1}) if $k>2\,\mathrm{V}_{\infty}\,C_4\,3^n.$ 
This proves the theorem.
\end{proof}

For a shrinking gradient Ricci soliton, the asymptotic volume ratio $\mathrm{V}_{\infty}$ is always finite.
By Li-Wang \cite{LW}, the following Sobolev inequality holds for dimension $n\geq 3.$

\begin{equation*}
\left( \int_{M}\phi ^{\frac{2n}{n-2}}\right) ^{\frac{n-2}{n}}\leq 
C(n)\,e^{-\frac{2\mu \left( g\right) }{n}}\int_{M}\left( \left\vert \nabla \phi
\right\vert ^{2}+S\phi ^{2}\right) 
\end{equation*}%
provided $\phi \in C_{0}^{\infty }(M).$ This implies Theorem \ref{T1intr}.

\begin{corollary}
Let $( M,g) $ be a gradient shrinking Ricci soliton with $\alpha<\infty,$ where
\begin{equation*}
\alpha =\limsup_{R\rightarrow \infty } \frac{1}{\mathrm{V}_{p}(R) }
\int_{B_{p}\left( R\right)}\left( S\,r^2\right) ^{\frac{n-1}{2}}.
\end{equation*}
Then the number of ends of $M$ is bounded
from above by $\Gamma( n,\alpha ,\mu (g)),$ a constant depending only on
dimension $n,$ $\mu(g) $ and $\alpha.$ 
\end{corollary}

For a submanifold $M$ in Euclidean space $\mathbb{R}^N,$ the well known Michael-Simon inequality \cite{A, MS}  states that
\begin{equation*}
\left( \int_{M}\left\vert \phi \right\vert^{\frac{n}{n-1}}\right) ^{\frac{n-1}{n}}\leq
C(n)\int_{M}\left( \left\vert \nabla \phi \right\vert +\left\vert
H\right\vert \left\vert \phi \right\vert\right)  
\end{equation*}%
for any $\phi \in C_{0}^{\infty }(M),$ where $H$ is the mean curvature vector of $M.$
By Theorem 2.1, we have the following conclusion.

\begin{corollary}
Let $M^n$ be a complete submanifold of $\mathbb{R}^N$ with $n\geq 2.$
Suppose
\begin{equation*}
\tilde{\alpha}=\limsup_{R\rightarrow \infty }\frac{1}{\mathrm{V}_{p}(R) }\int_{B_p\left(R\right) } \left(r\,\left\vert H\right\vert\right)^{n-1}<\infty 
\end{equation*}
and
\begin{equation*}
\mathrm{V}_{\infty }=\limsup_{R\rightarrow \infty }\frac{\mathrm{V}_p(R) }{R^{n}}<\infty. 
\end{equation*}
Then the number of ends of $M$ is bounded above by a constant $\Gamma$
depending only on the dimension $n,$ $\tilde{\alpha}$ and $V_{\infty }.$
\end{corollary}

Recall that a hypersurface $M\subset \mathbb{R}^{n+1}$ is a self shrinker of
the mean curvature flow if it satisfies the equation 
\begin{equation*}
H=\frac{1}{2}\left\langle x,\mathbf{n}\right\rangle ,
\end{equation*}%
where $x$ is the position vector, $H$ the mean curvature and $\mathbf{n}$
the unit normal vector. Self shrinkers arise naturally in the singularity
analysis of mean curvature flow. In fact, it follows from the monotonicity
formula of Huisken \cite{H1} that tangent flows at singularities of the mean
curvature flow are self shrinkers. Many examples have been constructed by
gluing methods by Kapouleas, Kleene, and M\"{o}ller in \cite{KKM} and Nguyen
in \cite{Nu}.

A self shrinker $M$ is asymptotically
conical if there exists a regular cone $\mathcal{C}\subset \mathbb{R}^{n+1}$
with vertex at the origin such that the rescaled submanifold $\lambda M$
converges to $\mathcal{C}$ locally smoothly as $\lambda \rightarrow 0.$ By a
theorem of Wang \cite{W}, the limiting cone $\mathcal{C}$ uniquely
determines the shrinker $M$.

For an asymptotically conical shrinker, clearly both $\tilde{\alpha}$ and
$\mathrm{V}_{\infty }$ are finite.

\begin{corollary}
Assume that $M^{n}\subset \mathbb{R}^{n+1}$ is an
asymptotically conical self shrinker of the mean curvature flow of dimension 
$n\geq 2.$ Then the number of ends $e( M) \leq \Gamma( n, \mathrm{V}_{\infty},\tilde{\alpha}),$ 
where $\tilde{\alpha}$ is defined in Corollary 2.3.
\end{corollary}

We would also like mention a recent result of Sun-Wang \cite{SW} which
bounds $e(M) $ in terms of the entropy and genus when $n=2.$ 

\section{Ends and solutions to Schr\"odinger equations\label{Ends}}

In this section we prove Theorem \ref{Eintr}. The standing assumption in
this section is that $M$ is complete and that $\sigma$ is a nonnegative, but
not identically zero, smooth function on $M.$

We first recall an interior gradient estimate for positive solution $u$ of $%
\Delta u=\sigma u$ established by Cheng and Yau (see Theorem 6 in \cite{CY}).

\begin{lemma}
\label{CY} Suppose $u>0$ is a solution to $\Delta u=\sigma u$ on the geodesic
ball $B_p(2r)$ centered at $p\in M$ and of radius $2r.$ Then 
\begin{equation*}
\left\vert \nabla \ln u\right\vert \leq C(r) \text{ on } B_p(r),
\end{equation*}
where $C(r)$ is a constant depending on $r,$ $\sigma$ and the Ricci
curvature lower bound of $M$ on $B_p(2r).$
\end{lemma}

In particular, the lemma implies that on any compact subset $K$ of $B_p(r)$%
, the Harnack inequality $u(x)\leq C(K)\,u(y)$ holds for $x, y\in K$ with a
constant $C(K)$ independent of $u.$

We now construct nontrivial solutions of the equation $\Delta u=\sigma u$
when $M$ has more than one end. In contrast to \cite{LT}, there is no need
to distinguish the two cases of $M$ being parabolic or nonparabolic.

\begin{theorem}
\label{E} Let $\left( M,g\right) $ be a complete manifold and $E_{1},
E_{2},\cdots, E_{l}$ the ends of $M$ with respect to the geodesic ball $%
B_p(r_0)$ with $l\geq 2.$ Then for each end $E_i,$ there exists a positive
solution $u_i$ to the equation $\Delta u_i=\sigma u_i $ on $M$ satisfying $%
0<u_i\leq 1$ on $M\backslash E_i$ and 
\begin{equation*}
\sup_{M}u_i=\limsup_{x\rightarrow E_i\left( \infty \right) }u_i\left(
x\right) >1.
\end{equation*}
Moreover, the functions $u_{1}, \cdots, u_{l}$ are linearly independent.
\end{theorem}

\begin{proof}
We first construct the functions $u_i.$ To ease notation, let $E=E_i$ and $%
F=F_i=M\backslash E_i.$ As $l\geq 2,$ $F$ must be unbounded. For $R\geq r_0,$
denote $E(R) =E\cap B_p(R)$ and $F( R) =F\cap B_p( R).$ Let $v_{R}: B_p( R)
\rightarrow \mathbb{R}$ be the solution of the Dirichlet problem 
\begin{align*}
\Delta v_{R} &=\sigma v_{R}\text{ \ in }B_p(R) \\
v_{R} &=0\text{ \ on }\partial F( R) \\
v_{R} &=1\text{ on }\partial E( R).
\end{align*}%
Since $\sigma \geq 0$ on $M,$ by the strong maximum principle, it follows
that $0<v_{R}<1$ in $B_p( R).$ We now normalize $v_{R}$ by setting 
\begin{equation*}
u_{R}=C_{R}\,v_{R},
\end{equation*}%
where 
\begin{equation*}
C_{R}=\left( \max_{B_p( r_{0}) }v_{R}\right) ^{-1}>1.
\end{equation*}%
Then $u_{R}$ is a solution of 
\begin{align*}
\Delta u_{R} &=\sigma u_{R}\text{ \ in }B_p( R) \\
u_{R} &=0\text{ \ on }\partial F( R) \\
u_{R} &=C_{R}\text{ on }\partial E( R).
\end{align*}%
In addition, 
\begin{equation}
\max_{B_p\left( r_{0}\right) }u_{R}=1.  \label{m7}
\end{equation}%
Hence, by Lemma \ref{CY} and the remark following it, we conclude from (\ref%
{m7}) that for any fixed $0<r<\frac{R}{2},$ 
\begin{equation*}
\sup_{B_p( r) }u_{R}\leq C\left(r\right)
\end{equation*}%
and 
\begin{equation*}
\sup_{B_p( r) }\left\vert \nabla u_{R}\right\vert \leq C(r),
\end{equation*}%
where $C\left( r\right) $ is a constant independent of $R.$ It is now easy
to see that a subsequence of $u_{R}$ converges to a solution $u>0$ of $%
\Delta u=\sigma u$ on $M.$ Note that $u$ can not be a constant function as $%
\sigma$ is not identically $0.$

Since $u_{R}=0$ on $\partial F(R),$ the strong maximum principle implies
that $\sup_{\partial E\left( r\right) }u_{R}$ is strictly increasing in $r$
and $\sup_{\partial F\left( r\right) }u_{R}$ decreasing in $r.$ Therefore,
the same holds true for the function $u.$ In particular, by the fact that 
\begin{equation}
\max_{B_p\left( r_{0}\right) }u=1,
\end{equation}%
one concludes that $0<u\leq 1$ on $F=M\backslash E$ and 
\begin{equation*}
\sup_{M}u=\limsup_{x\rightarrow E\left( \infty \right) }u\left( x\right) >1.
\end{equation*}
This finishes our construction of the function $u_i.$

We now turn to prove that the functions $u_{1}, \cdots, u_{l}$ are linearly
independent. Assume that 
\begin{equation}
\sum_{j=1}^{l}a_{j}u_{j}=0  \label{m20}
\end{equation}%
for some constants $a_{j}\in \mathbb{R}.$ For an arbitrary but fixed $j,$ if 
$u_{j}$ is unbounded on $E_{j},$ then clearly $a_{j}=0$ as $u_{i}$ is
bounded on $E_{j}$ for all $i\neq j.$

So we may assume from here on that each $u_{j}$ is bounded on $E_{j}.$ Let 
\begin{equation*}
S_{j}=\sup_{E_{j}}u_{j}>1.
\end{equation*}%
Then there exists a sequence $x_{j,k}\in E_{j}$ such that 
\begin{equation}
\lim_{k\rightarrow \infty }\left( S_{j}-u_{j}\right) \left( x_{j,k}\right)
=0 .  \label{m21}
\end{equation}%
Note that $S_{j}-u_{j}>0$ on $M.$ In particular, there exists a constant $%
C_{j}>0$ satisfying $S_{j}-u_{j}>\frac{1}{C_{j}}$ on $B_p(r_{0}).$ We now
claim that for $i\neq j,$ 
\begin{equation}
u_{i}\leq C_{j}\left( S_{j}-u_{j}\right)  \label{m22}
\end{equation}%
on $E_{j}.$

Indeed, recall from the construction that $u_{i}$ is the limit of a
subsequence of $u_{i,R}$ satisfying 
\begin{align*}
\Delta u_{i,R} &=\sigma u_{i,R}\text{ \ in }B_p( R) \\
u_{i,R} &=0\text{ \ on }\partial F_{i}( R) \\
u_{i,R} &=C_{i,R}\text{ on }\partial E_{i}( R),
\end{align*}%
where $F_{i}=M\backslash E_{i},$ together with 
\begin{equation*}
\max_{B_p\left( r_{0}\right) }u_{i,R}=1.
\end{equation*}%
Now the function 
\begin{equation*}
w_{i,R}=u_{i,R}-C_{j}\left( S_{j}-u_{j}\right)
\end{equation*}%
satisfies $\Delta w_{i,R}\geq 0$ on $F_{i}( R) \setminus F_{i}( r_{0})$ as $%
\sigma\geq 0.$ Also, $w_{i,R}<0$ on $\partial F_{i}( R) \cup \partial F_{i}(
r_{0}).$ By the maximum principle, $w_{i,R}<0$ on $F_{i}( R) \backslash
F_{i}\left( r_{0}\right).$ After taking limit, one concludes that $u_i\leq
C_{j}\left( S_{j}-u_{j}\right)$ on $F_{i} \backslash F_{i}\left(
r_{0}\right).$ Since $i\neq j$ and $E_j\subset F_{i} \backslash F_{i}\left(
r_{0}\right),$ the claim follows.

By (\ref{m21}) and (\ref{m22}) it follows that 
\begin{equation*}
\lim_{k\rightarrow \infty }u_{i}\left( x_{j,k}\right) =\left\{ 
\begin{array}{c}
0 \\ 
S_{j}%
\end{array}%
\right. 
\begin{array}{c}
\text{if }i\neq j \\ 
\text{if }i=j.%
\end{array}%
\end{equation*}%
Plugging this into (\ref{m20}), one infers that $a_{j}=0.$ But $j$ is
arbitrary. This proves that $u_{1}, \cdots, u_{l}$ are linearly independent.
\end{proof}

\section{Growth estimates\label{GR}}

Our focus in this section is on growth rate estimates for positive solutions
to $\Delta u=\sigma u.$ We fix a large enough positive constant $R_{0}$ and
assume that the manifold $M$ admits a proper function $\rho $ satisfying 
\begin{equation}
\frac{1}{2}\leq \left\vert \nabla \rho \right\vert \leq 1\text{ \ and \ }%
\Delta \rho \leq \frac{m}{\rho }  \label{L2}
\end{equation}%
in the weak sense for $\rho \geq R_{0},$ where $m$ is a positive constant.
Denote the sublevel and level set of $\rho $ by 
\begin{equation*}
D(r)=\left\{ x\in M:\rho \left( x\right) <r\right\} \text{\ and\ }\Sigma
(r)=\left\{ x\in M:\rho (x)=r\right\}
\end{equation*}%
respectively. They are compact as $\rho $ is proper.
Denote with $\mathrm{V}(r)$ the volume of $D(r)$ and with $\mathrm{A}(r)$ the area of $\Sigma(r)$. 

\begin{definition}
A manifold $\left( M,g\right) $ has the mean value property $\left( \mathcal{M}%
\right)$ if there exist constants $A_{0}>0$ and $\nu >1$ such that for any $%
0<\theta \leq 1$ and $R\geq 4R_0,$ 
\begin{equation}
\sup_{\Sigma \left( R\right) }u\leq \frac{A_{0}}{\theta ^{2\nu }}\frac{1}{%
\mathrm{V}( ( 1+\theta ) R) }\int_{D\left( \left( 1+\theta \right) R\right)
\backslash D\left( R_{0}\right) }u  \label{mean}
\end{equation}%
holds true for any function $u>0$ satisfying $\Delta u\geq \sigma u$ on $D(
2R)\backslash D( R_{0}).$
\end{definition}

We begin with a simple observation. Integrating by parts, one immediately
sees that for any $C^{1}$ function $w$ and $r\geq R_{0},$ 
\begin{equation*}
\int_{D\left( r\right) }w\Delta \rho +\int_{D\left( r\right) }\left\langle
\nabla w,\nabla \rho \right\rangle =\int_{\Sigma \left( r\right) }w\frac{%
\partial \rho }{\partial \eta }
\end{equation*}%
where $\eta$ is the unit normal vector to $\Sigma( r)$ given by $\eta=\frac{%
\nabla \rho }{\left\vert \nabla \rho \right\vert }.$ Taking a derivative in $%
r$ of this identity yields the following formula: 
\begin{equation}
\frac{d}{dr}\int_{\Sigma \left( r\right) }w\left\vert \nabla \rho
\right\vert =\int_{\Sigma \left( r\right) }\frac{\left\langle \nabla
w,\nabla \rho \right\rangle }{\left\vert \nabla \rho \right\vert }%
+\int_{\Sigma \left( r\right) }\frac{w\Delta \rho }{\left\vert \nabla \rho
\right\vert }\ .  \label{k0}
\end{equation}

The following lemma provides volume and area estimates.

\begin{lemma}
\label{Area} Let $A(r)$ be the area of $\Sigma(r)$ and $V(r)$ the volume of $%
D(r).$ Then 
\begin{align*}
\mathrm{A}( r) &\leq \frac{c(m)}{r}\mathrm{V}( r), \\
\mathrm{V}(( 1+\theta)r) &\leq (1+\theta ) ^{c( m)}\mathrm{V}(r), \\
\mathrm{V}(r) &\leq r^{\gamma( m)}\mathrm{V}(R_0)
\end{align*}%
for all $r\geq R_{0}$ and $0<\theta \leq 1,$ where $c(m)$ and $\gamma(m)$
depend only on $m$.
\end{lemma}

\begin{proof}
By the co-area formula, there exists $\frac{r}{2}<t<r$ such that 
\begin{align}
\mathrm{V}( r) &\geq \mathrm{Vol}\left(D(r) \backslash D\left( \frac{r}{2}%
\right) \right)  \label{k3} \\
&=\frac{r}{2}\int_{\Sigma(t)}\frac{1}{\left\vert \nabla \rho \right\vert }. 
\notag
\end{align}%
From (\ref{L2}) we have 
\begin{equation*}
\Delta \rho \leq \frac{4m}{\rho}\left\vert \nabla \rho \right\vert ^{2}
\end{equation*}%
for all $r\geq R_{0}.$ Hence, applying (\ref{k0}) with $w=1$ implies 
\begin{align*}
\frac{d}{dr}\int_{\Sigma \left( r\right) }\left\vert \nabla \rho \right\vert
&=\int_{\Sigma \left( r\right) }\frac{\Delta \rho }{\left\vert \nabla \rho
\right\vert } \\
&\leq \frac{4m}{r}\int_{\Sigma \left( r\right) }\left\vert \nabla \rho
\right\vert.
\end{align*}%
Integrating in $r$ we conclude that 
\begin{align*}
\int_{\Sigma \left( r\right) }\left\vert \nabla \rho \right\vert &\leq
\left( \frac{r}{t}\right) ^{4m}\int_{\Sigma \left( t\right) }\left\vert
\nabla \rho \right\vert \\
&\leq \left( \frac{r}{t}\right) ^{4m}\int_{\Sigma \left( t\right) }\frac{1}{%
\left\vert \nabla \rho \right\vert }.
\end{align*}%
Together with (\ref{k3}), this implies 
\begin{equation}
\int_{\Sigma \left( r\right) }\left\vert \nabla \rho \right\vert \leq \frac{%
c(m) }{r}\mathrm{V}\left( r\right).  \label{k4}
\end{equation}%
Now the area estimate follows from (\ref{L2}).

Note that (\ref{k4}) and (\ref{L2}) also imply

\begin{equation*}
\mathrm{V}^{\prime }\left( r\right) \leq \frac{c(m) }{r}\mathrm{V}( r).
\end{equation*}%
Integrating in $r$ we obtain 
\begin{equation}
\mathrm{V}( R) \leq \left( \frac{R}{r}\right) ^{c(m) }\mathrm{V}(r)
\label{k5}
\end{equation}%
for all $R_{0}<r<R.$ Clearly, it gives both the volume doubling property and
growth estimate. This proves the result.
\end{proof}

The next lemma is our starting point for establishing growth estimates for
positive solutions to $\Delta u=\sigma u.$

\begin{lemma}
\label{I} A positive solution $u$ of $\Delta u=\sigma u$ on $D(R)\backslash
D(R_{0})$ satisfies 
\begin{equation*}
\frac{d}{dr}\left( \frac{1}{r^{4m}}\int_{\Sigma \left( r\right) }u\left\vert
\nabla \rho \right\vert \right) \leq \frac{1}{r^{4m}}\int_{D\left( r\right)
\backslash D\left( r_{0}\right) }\sigma u+\frac{1}{r^{4m}}\int_{\Sigma
\left( r_{0}\right) }\frac{\left\langle \nabla u,\nabla \rho \right\rangle }{%
\left\vert \nabla \rho \right\vert }
\end{equation*}%
for all $R\geq r\geq r_{0}\geq R_{0}.$
\end{lemma}

\begin{proof}
Applying (\ref{k0}) to $w=u$ and taking into account that

\begin{equation*}
\int_{\Sigma \left( r\right) }\frac{\left\langle \nabla u,\nabla \rho
\right\rangle }{\left\vert \nabla \rho \right\vert }=\int_{D\left( r\right)
\backslash D\left( r_{0}\right) }\Delta u+\int_{\Sigma \left( r_{0}\right) }%
\frac{\left\langle \nabla u,\nabla \rho \right\rangle }{\left\vert \nabla
\rho \right\vert },
\end{equation*}%
we obtain 
\begin{equation}
\frac{d}{dr}\int_{\Sigma \left( r\right) }u\left\vert \nabla \rho
\right\vert =\int_{D\left( r\right) \backslash D\left( r_{0}\right) }\sigma
u+\int_{\Sigma \left( r\right) }u\frac{\Delta \rho }{\left\vert \nabla \rho
\right\vert }+\int_{\Sigma \left( r_{0}\right) }\frac{\left\langle \nabla
u,\nabla \rho \right\rangle }{\left\vert \nabla \rho \right\vert }.
\label{k2}
\end{equation}%
By (\ref{L2}) we have that 
\begin{equation*}
\int_{\Sigma \left( r\right) }u\frac{\Delta \rho }{\left\vert \nabla \rho
\right\vert } \leq \frac{m}{r}\int_{\Sigma \left( r\right) }\frac{u}{%
\left\vert \nabla \rho \right\vert } \leq \frac{4m}{r}\int_{\Sigma \left(
r\right) }u\left\vert \nabla \rho \right\vert
\end{equation*}
for $r\geq r_{0}\geq R_{0}.$ Plugging this into \eqref{k2} implies 
\begin{equation*}
\frac{d}{dr}\int_{\Sigma \left( r\right) }u\left\vert \nabla \rho
\right\vert \leq \int_{D\left( r\right) \backslash D\left( r_{0}\right)
}\sigma u+\frac{4m}{r}\int_{\Sigma \left( r\right) }u\left\vert \nabla \rho
\right\vert +\int_{\Sigma \left( r_{0}\right) }\frac{\left\langle \nabla
u,\nabla \rho \right\rangle }{\left\vert \nabla \rho \right\vert }.
\end{equation*}
This proves the result.
\end{proof}

We now prove a preliminary growth estimate by imposing a pointwise quadratic
decay assumption on $\sigma $ of the form 
\begin{equation}
\sigma \leq \frac{\Upsilon }{\rho ^{2}}\text{ \ on }M\backslash D(r_{0}),
\label{T}
\end{equation}%
where $r_{0}\geq 4R_{0}$ and $\Upsilon >0$ is a constant.

\begin{proposition}
\label{G} Assume that $\left( M,g\right)$ admits a proper function $\rho$
satisfying (\ref{L2}) and has the mean value property $\left( \mathcal{M}%
\right).$ If $\sigma$ decays quadratically as in (\ref{T}), then there
exists a constant $C=C(m, \Upsilon )>0$ such that 
\begin{equation*}
u\leq \left( \rho +1\right) ^{C}\sup_{D\left( r_{0}\right)\backslash D(R_0) }u\text{ \ on } D%
\bigg( \frac{R}{2}\bigg)\backslash D(R_0)
\end{equation*}
for any positive solution of $\Delta u=\sigma u$ on $D\left( R\right)\backslash D(R_0)$ with $%
R\geq r_{0}.$
\end{proposition}

\begin{proof}
The result is obvious if $R\leq 2r_{0}.$ Hence, we may assume from now on
that $R>2r_{0}.$ By Lemma \ref{CY}, it follows that there exists $C( r_{0})
>0$ such that 
\begin{equation}
\left\vert \int_{\Sigma \left( r_{0}\right) }\frac{\left\langle \nabla
u,\nabla \rho \right\rangle }{\left\vert \nabla \rho \right\vert }%
\right\vert \leq C( r_{0}) \sup_{\Sigma \left( r_{0}\right) }u  \label{m0.0}
\end{equation}%
with the constant $C( r_{0})$ independent of $u.$

By normalizing $u$ if necessary, we may assume that 
\begin{equation}
\sup_{D\left( r_{0}\right)\backslash D(R_0) }u=1.  \label{m0}
\end{equation}%
So we get 
\begin{equation}
\left\vert \int_{\Sigma \left( r_{0}\right) }\frac{\left\langle \nabla
u,\nabla \rho \right\rangle }{\left\vert \nabla \rho \right\vert }%
\right\vert \leq C(r_{0}).  \label{m0.1}
\end{equation}%
By Lemma \ref{I} and (\ref{L2}) we have that 
\begin{equation}  \label{m1}
\begin{split}
\frac{d}{dr}\left( \frac{1}{r^{4m}}\int_{\Sigma \left( r\right) }u\left\vert
\nabla \rho \right\vert \right) &\leq \frac{1}{r^{4m}}\int_{D\left( r\right)
\backslash D\left( r_{0}\right) }\sigma u+\frac{1}{r^{4m}}\int_{\Sigma
\left( r_{0}\right) }\frac{\left\langle \nabla u,\nabla \rho \right\rangle }{%
\left\vert \nabla \rho \right\vert } \\
&\leq \frac{4}{r^{4m}}\int_{D\left( r\right) \backslash D\left( r_{0}\right)
}\sigma u\left\vert \nabla \rho \right\vert ^{2}+\frac{1}{r^{4m}}%
\int_{\Sigma \left( r_{0}\right) }\frac{\left\langle \nabla u,\nabla \rho
\right\rangle }{\left\vert \nabla \rho \right\vert }
\end{split}%
\end{equation}
for all $r\in \left[ r_{0},R\right]$.

Combining (\ref{m1}), (\ref{m0.1}) and (\ref{T}), we conclude 
\begin{equation}
\frac{d}{dr}\left( \frac{1}{r^{4m}}\int_{\Sigma \left( r\right) }u\left\vert
\nabla \rho \right\vert \right) \leq \frac{4\Upsilon }{r^{4m}}\int_{D\left(
r\right) \backslash D\left( r_{0}\right) }u\frac{\left\vert \nabla \rho
\right\vert ^{2}}{\rho ^{2}}+\frac{C( r_{0}) }{r^{4m}}  \label{m2}
\end{equation}%
for all $r\in \left[ r_{0},R\right].$ If we set 
\begin{equation}
\omega \left( r\right) =\int_{D\left( r\right) \backslash D\left(
r_{0}\right) }u\frac{\left\vert \nabla \rho \right\vert ^{2}}{\rho ^{2}},
\label{m3}
\end{equation}%
then the co-area formula gives 
\begin{equation*}
\omega ^{\prime }\left( r\right) =\frac{1}{r^{2}}\int_{\Sigma \left(
r\right) }u\left\vert \nabla \rho \right\vert.
\end{equation*}%
So (\ref{m2}) becomes 
\begin{equation*}
\frac{d}{dr}\left( \frac{1}{r^{4m-2}}\omega ^{\prime }\left( r\right)
\right) \leq \frac{4\Upsilon }{r^{4m}}\omega \left( r\right) +\frac{C(
r_{0}) }{r^{4m}}
\end{equation*}%
or 
\begin{equation}
r^{2}\omega ^{\prime \prime }\left( r\right) -\left( 4m-2\right) r\omega
^{\prime }\left( r\right) -4\Upsilon \omega \left( r\right) \leq C( r_{0})
\label{m4}
\end{equation}%
for all $\ r\in \left[ r_{0},R\right].$ Direct calculation then implies that
the function 
\begin{equation}
\xi \left( r\right) =r^{a}\omega(r)  \label{m4''}
\end{equation}%
satisfies 
\begin{equation}
r\xi ^{\prime \prime }\left( r\right) -\left( 2a+4m-2\right) \xi ^{\prime
}\left( r\right) \leq C\left( r_{0}\right) r^{a-1}  \label{m5}
\end{equation}%
for all $r\in \left[ r_{0},R\right],$ where 
\begin{equation}
a=\frac{\sqrt{\left( 4m-1\right) ^{2}+16\Upsilon }-\left( 4m-1\right) }{2}.
\label{m4'}
\end{equation}%
Rewriting (\ref{m5}) into 
\begin{equation*}
\frac{d}{dr}\left( \frac{\xi ^{\prime }\left( r\right) }{r^{2a+4m-2}}\right)
\leq \frac{C( r_{0}) }{r^{a+4m}}
\end{equation*}%
and integrating from $r_{0}$ to $r,$ we get 
\begin{equation}
\xi ^{\prime }\left( r\right) \leq \left( \frac{r}{r_{0}}\right)
^{2a+4m-2}\xi ^{\prime }( r_{0}) +C( r_{0}) r^{2a+4m-2}  \label{m6}
\end{equation}%
for all $r\in \left[ r_{0},R\right].$

According to (\ref{m4''}) and (\ref{m3}) we have 
\begin{equation*}
\xi ^{\prime }( r_{0}) =r_{0}^{a-2}\int_{\Sigma \left( r_{0}\right)
}u\left\vert \nabla \rho \right\vert.
\end{equation*}%
Hence, by (\ref{m0}), 
\begin{equation*}
\xi^{\prime }( r_{0}) \leq C( r_{0}).
\end{equation*}%
Plugging into (\ref{m6}) we conclude that 
\begin{equation*}
\xi ^{\prime }( r) \leq C( r_{0}) r^{2a+4m-2}
\end{equation*}%
for all $r\in \left[ r_{0},R\right].$ After integrating from $r_{0}$ to $r,$
this immediately leads to 
\begin{equation*}
\omega(r) \leq C( r_{0}) r^{a+4m-1}.
\end{equation*}%
In view of (\ref{m3}) and (\ref{m4'}), we have 
\begin{equation*}
\int_{D\left( r\right)\backslash D(R_0) }u\leq C(r_{0}) r^{C\left( m,\Upsilon \right) }
\end{equation*}%
for all $r\in \left[ r_{0},R\right].$ Finally, the mean value property
implies that 
\begin{equation*}
\sup_{\Sigma( \frac{1}{2}r) }u\leq C( A,\mu ,r_{0}) r^{C\left( m,\Upsilon
\right) }
\end{equation*}%
for all $r\in \left[ 2r_{0},R\right].$ This proves the result.
\end{proof}

We remark that the assumption of $\sigma$ being of quadratic decay is
optimal in the sense that any slower decay will render the result to fail.
Indeed, on Euclidean space, the function $u(x)=\exp\left({r^{\epsilon}(x)}%
\right)$ satisfies the equation $\Delta u=\sigma u$ with $\sigma$ decaying
of order $2-2\epsilon$.

Our main result of this section is that the order of polynomial growth of $u$
in fact only depends on an integral quantity of the function $\sigma$
provided that $u$ is a priori of polynomial growth, namely, 
\begin{equation*}
\left\vert u\right\vert \leq \rho ^{C}\text{ \ on }M\backslash D(R_{0})
\end{equation*}%
for some constant $C>0.$

In the following, we denote 
\begin{equation*}
\alpha =\limsup_{R\rightarrow \infty }\left( R^{2q}\fint_{\Sigma \left(
R\right) }\sigma ^{q}\right) ^{\frac{1}{q}}
\end{equation*}%
with $q\geq 1$ to be specified.

\begin{theorem}
\label{Growth}Assume that $\left( M,g\right)$ admits a proper function $\rho$
satisfying (\ref{L2}) and has the mean value property $\left( \mathcal{M}%
\right).$ For a positive function $u$ of polynomial growth, satisfying $%
\Delta u=\sigma u$ on $M\backslash D\left( R_{0}\right)$, if $\alpha<\infty$
for some $q>\nu -\frac{1}{2}$, then there exists a constant $\Gamma( m,
A_{0},\nu, \alpha)>0$ such that 
\begin{equation*}
u\leq \Lambda \left( \rho ^{\Gamma}+1\right) \text{ \ on }M\backslash D(
R_{0}),
\end{equation*}%
where $\Lambda>0$ is a constant depending on $u.$ The same estimate for $u$
holds true in the case $q=\nu -\frac{1}{2}$ with $\Gamma=\Gamma(m, A_0, \nu)$
provided that $\alpha\leq \alpha _{0}\left( m, A_{0}, \nu \right),$ a
sufficiently small positive constant.
\end{theorem}

\begin{proof}
By the H\"{o}lder inequality, $\alpha$ is increasing in $q.$ So we may
restrict our attention to those $q$ that

\begin{equation*}
0\leq \varepsilon <\frac{1}{2},
\end{equation*}
where

\begin{equation}
\varepsilon =\frac{2q+1-2\nu }{q}.  \label{eps}
\end{equation}%
To treat both cases $q>\nu -\frac{1}{2}$ and $q=\nu -\frac{1}{2}$ at the
same time, we let 
\begin{equation}
\bar{\alpha}=\min \left\{ \alpha ,1\right\} \text{ \ and }\widetilde{\alpha }%
=\max \left\{ \alpha ,1\right\} .  \label{n1.1}
\end{equation}%
Note that $\alpha =\bar{\alpha}\,\widetilde{\alpha }.$ In the following, 
\begin{equation}
C_{0}=C_{0}\left( m,A_{0},\nu ,\widetilde{\alpha }\right) >1  \label{C0}
\end{equation}%
is a fixed large constant, depending only on $m,A_{0},\nu $ and $\widetilde{%
\alpha },$ to be specified later.

In view of the definition of $\alpha,$ there exists $r_{0}\geq 4R_{0}$ such
that 
\begin{equation*}
\int_{\Sigma \left( r\right) }\frac{\sigma ^{q}}{\left\vert \nabla \rho
\right\vert }\leq 3\alpha ^{q}r^{-2q}\mathrm{A}( r)
\end{equation*}%
for all $r\geq r_{0}.$ From Lemma \ref{Area} it follows that 
\begin{equation}
\int_{\Sigma \left( r\right) }\frac{\sigma ^{q}}{\left\vert \nabla \rho
\right\vert }\leq c\left( m\right) \alpha ^{q}r^{-2q-1}\mathrm{V}( r) ,
\label{al1}
\end{equation}%
for all $r\geq r_{0}.$

Denote 
\begin{equation}
\chi(r) =\int_{D( r) \backslash D\left( R_{0}\right) }u\frac{\left\vert
\nabla \rho \right\vert ^{2}}{\rho ^{4m}}.  \label{kai}
\end{equation}%
We claim that $\chi $ satisfies the following inequality. 
\begin{equation}
r^{4m}\chi ^{\prime \prime }(r) \leq \frac{C_{0}\bar{\alpha}}{\theta ^{\frac{%
2\nu }{q}}}\int_{r_{0}}^{r}\chi^{\frac{1}{q}}(( 1+\theta) t)( \chi ^{\prime
}( t)) ^{1-\frac{1}{q}}t^{4m-2-\frac{1}{q}}dt+\Lambda _{0}  \label{id}
\end{equation}%
for all $r\geq r_{0}$ and $0<\theta \leq 1,$ where 
\begin{equation}
\Lambda _{0}=\int_{\Sigma \left( r_{0}\right) }\left( u+\left\vert \nabla
u\right\vert \right).  \label{Lambda}
\end{equation}

We first prove (\ref{id}) for $q>1.$ By the co-area formula, 
\begin{equation}
\chi ^{\prime }(r)=\frac{1}{r^{4m}}\int_{\Sigma (r)}u\left\vert \nabla \rho
\right\vert .  \label{kai'}
\end{equation}%
Hence, using Lemma \ref{I}, we have 
\begin{equation}
\begin{split}
\chi ^{\prime \prime }(r)& =\frac{d}{dr}\left( \frac{1}{r^{4m}}\int_{\Sigma
\left( r\right) }u\left\vert \nabla \rho \right\vert \right)  \label{p1} \\
& \leq \frac{1}{r^{4m}}\int_{D\left( r\right) \backslash D\left(
r_{0}\right) }\sigma u+\frac{1}{r^{4m}}\int_{\Sigma \left( r_{0}\right) }%
\frac{\left\langle \nabla u,\nabla \rho \right\rangle }{\left\vert \nabla
\rho \right\vert }.
\end{split}%
\end{equation}
The first term can be estimated by the co-area formula and H\"{o}lder
inequality as 
\begin{equation}
\begin{split}
\int_{D\left( r\right) \backslash D\left( r_{0}\right) }\sigma u&
=\int_{r_{0}}^{r}\left( \int_{\Sigma \left( t\right) }\frac{\sigma u}{%
\left\vert \nabla \rho \right\vert }\right) dt \\
& \leq \int_{r_{0}}^{r}\left( \int_{\Sigma \left( t\right) }\frac{\sigma ^{q}%
}{\left\vert \nabla \rho \right\vert }\right) ^{\frac{1}{q}}\left(
\int_{\Sigma \left( t\right) }\frac{u^{p}}{\left\vert \nabla \rho
\right\vert }\right) ^{\frac{1}{p}}dt,
\end{split}%
\end{equation}%
where 
\begin{equation*}
\frac{1}{p}+\frac{1}{q}=1.
\end{equation*}%
Invoking (\ref{al1}) we conclude 
\begin{equation}
\int_{D(r)\backslash D\left( r_{0}\right) }\sigma u\leq c(m)\alpha
\int_{r_{0}}^{r}\left( \int_{\Sigma \left( t\right) }\frac{u^{p}}{\left\vert
\nabla \rho \right\vert }\right) ^{\frac{1}{p}}\frac{\mathrm{V}(t)^{\frac{1}{%
q}}}{t^{2+\frac{1}{q}}}dt.  \label{p3}
\end{equation}%
On the other hand, the mean value property (\ref{mean}) implies 
\begin{equation}
\begin{split}
\sup_{\Sigma \left( t\right) }u& \leq \frac{A_{0}}{\theta ^{2\nu }}\frac{1}{%
\mathrm{V}((1+\theta )t)}\int_{D\left( \left( 1+\theta \right) t\right)
\backslash D\left( R_{0}\right) }u \\
& \leq \frac{4A_{0}}{\theta ^{2\nu }}\frac{\left( \left( 1+\theta \right)
t\right) ^{4m}}{\mathrm{V}(t)}\int_{D\left( \left( 1+\theta \right) t\right)
\backslash D\left( R_{0}\right) }u\frac{\left\vert \nabla \rho \right\vert
^{2}}{\rho ^{4m}} \\
& \leq \frac{c(m)A_{0}}{\theta ^{2\nu }}\frac{t^{4m}}{\mathrm{V}(t)}\chi
((1+\theta )t)
\end{split}%
\end{equation}%
for all $t\geq r_{0}.$ Therefore, 
\begin{align*}
\left( \int_{\Sigma \left( t\right) }\frac{u^{p}}{\left\vert \nabla \rho
\right\vert }\right) ^{\frac{1}{p}}& \leq \left( \sup_{\Sigma \left(
t\right) }u\right) ^{\frac{1}{q}}\left( \int_{\Sigma \left( t\right) }\frac{u%
}{\left\vert \nabla \rho \right\vert }\right) ^{\frac{1}{p}} \\
& \leq \frac{c(m)A_{0}^{\frac{1}{q}}}{\theta ^{\frac{2\nu }{q}}}\frac{t^{%
\frac{4m}{q}}}{\mathrm{V}(t)^{\frac{1}{q}}}\chi ^{\frac{1}{q}}((1+\theta
)t)\left( \int_{\Sigma \left( t\right) }\frac{u}{\left\vert \nabla \rho
\right\vert }\right) ^{\frac{1}{p}} \\
& \leq \frac{c(m)A_{0}^{\frac{1}{q}}}{\theta ^{\frac{2\nu }{q}}}\frac{t^{4m}%
}{\mathrm{V}\left( t\right) ^{\frac{1}{q}}}\chi ^{\frac{1}{q}}((1+\theta
)t)\left( \chi ^{\prime }\left( t\right) \right) ^{\frac{1}{p}},
\end{align*}%
where in the last line we have used (\ref{kai'}).

Plugging this into (\ref{p3}) we conclude that 
\begin{equation}
\int_{D\left( r\right) \backslash D\left( r_{0}\right) }\sigma u\leq \frac{%
C_{0}\bar{\alpha}}{\theta ^{\frac{2\nu }{q}}}\int_{r_{0}}^{r}\chi^{\frac{1}{q%
}}((1+\theta) t) \left( \chi ^{\prime }( t) \right) ^{\frac{1}{p}}t^{4m-2-%
\frac{1}{q}}dt,  \label{p4}
\end{equation}%
where $C_{0}=c( m) A_{0}^{\frac{1}{q}}\widetilde{\alpha }$ for some $c(m)$
depending only on $m.$

By (\ref{p1}) and (\ref{p4}) it follows that 
\begin{align*}
\chi ^{\prime \prime }( r) &\leq \frac{C_{0}\bar{\alpha}}{\theta ^{\frac{%
2\nu }{q}}r^{4m}}\int_{r_{0}}^{r}\chi^{\frac{1}{q}}(( 1+\theta ) t) \left(
\chi ^{\prime }\left( t\right) \right) ^{\frac{1}{p}}t^{4m-2-\frac{1}{q}}dt
\\
&+\frac{1}{r^{4m}}\int_{\Sigma \left( r_{0}\right) }\frac{\left\langle
\nabla u,\nabla \rho \right\rangle }{\left\vert \nabla \rho \right\vert }.
\end{align*}%
In view of (\ref{Lambda}), this can be rewritten into 
\begin{equation*}
r^{4m}\chi ^{\prime \prime }( r) \leq \frac{C_{0}\bar{\alpha}}{\theta ^{%
\frac{2\nu }{q}}}\int_{r_{0}}^{r}\chi^{\frac{1}{q}}(( 1+\theta ) t) \left(
\chi ^{\prime }\left( t\right) \right) ^{1-\frac{1}{q}}t^{4m-2-\frac{1}{q}%
}dt+\Lambda _{0}.
\end{equation*}%
Hence, (\ref{id}) holds for any $q>1.$

To extend the result to $q=1,$ we simply let $q\rightarrow 1$ in (\ref{id})
and note that both sides are continuous as functions of $q.$

In conclusion, we have 
\begin{equation}
r^{4m}\chi ^{\prime \prime }( r) \leq \frac{C_{0}\bar{\alpha}}{\theta ^{%
\frac{2\nu }{q}}}\int_{r_{0}}^{r}\chi ^{\frac{1}{q}}(( 1+\theta ) t) \left(
\chi ^{\prime }\left( t\right) \right) ^{1-\frac{1}{q}}t^{4m-2-\frac{1}{q}%
}dt+\Lambda _{0}  \label{id1}
\end{equation}%
for all $r\geq r_{0}$ and $0<\theta \leq 1$.

Since $u$ is assumed to be of polynomial growth, there exist constants $\bar{%
b}>0$ and $\bar{\Lambda}>0$ such that 
\begin{equation*}
u\leq \bar{\Lambda}\rho ^{\bar{b}}\text{ \ on }M\backslash D( r_{0}).
\end{equation*}%
Together with Lemma \ref{Area} we get 
\begin{equation*}
\chi^{\prime}( r) =\frac{1}{r^{4m}}\int_{\Sigma \left( r\right) }u\left\vert
\nabla \rho \right\vert \leq c( m) \bar{\Lambda}\,r^{\bar{b}+\gamma( m)}\mathrm{V}(R_0) .
\end{equation*}%
Therefore, for $r\geq r_{0},$ 
\begin{equation}
\chi ^{\prime }( r) \leq \Lambda r^{b}  \label{n1}
\end{equation}%
for some constants $b>0$ and $\Lambda >0.$

Obviously, the constant $b$ in (\ref{n1}) can be chosen in such a way that (%
\ref{n1}) no longer holds with $b$ replaced by $b-1$ for whatever constant $%
\Lambda .$ Also, the constant $\Lambda $ can be arranged to satisfy that $%
\Lambda \geq \Lambda _{0}$ and 
\begin{equation}
\Lambda \geq \int_{D\left( r_{0}\right) \backslash D( R_{0}) }( u+\left\vert
\nabla u\right\vert) .  \label{Lambda1}
\end{equation}

For $\varepsilon $ in (\ref{eps}) and $C_{0}=C_{0}\left( m,A_{0},\nu, 
\widetilde{\alpha }\right)$ from (\ref{id1}) we assume by contradiction that 
\begin{equation}
\min \left\{ \frac{b^{\varepsilon }}{\bar{\alpha}},b\right\} >\left(
100C_{0}\right) ^{2}.  \label{bC0}
\end{equation}

We now prove by induction on $k\geq 0$ that 
\begin{equation}
\chi ^{\prime }( r) \leq \Lambda \bigg( \Big( \frac{\bar{\alpha}}{%
b^{\varepsilon }}\Big) ^{\frac{k}{2}}r^{b}+r^{b-1}\bigg)  \label{n2}
\end{equation}%
for all $r\geq r_{0}.$

Clearly, (\ref{n2}) holds for $k=0$ in view of (\ref{n1}). We assume it is
true for $k$ and prove it for $k+1.$ Integrating (\ref{n2}) we obtain that 
\begin{align*}
\chi( r) &\leq \Lambda \int_{r_{0}}^{r}\bigg( \Big( \frac{\bar{\alpha}}{%
b^{\varepsilon }}\Big) ^{\frac{k}{2}}t^{b}+t^{b-1}\bigg) dt+\chi ( r_{0}) \\
&\leq \frac{\Lambda }{b}\bigg( \Big( \frac{\bar{\alpha}}{b^{\varepsilon }}%
\Big) ^{\frac{k}{2}}r^{b+1}+r^{b}\bigg) +\Lambda ,
\end{align*}%
where the last line follows from (\ref{Lambda1}). Since 
\begin{equation*}
\Lambda \leq \frac{\Lambda }{b}r^{b},
\end{equation*}%
this implies 
\begin{equation*}
\chi ( r) \leq \frac{2\Lambda }{b}\bigg( \Big( \frac{\bar{\alpha}}{%
b^{\varepsilon }}\Big) ^{\frac{k}{2}}r^{b+1}+r^{b}\bigg)
\end{equation*}%
for all $r\geq r_{0}.$ Therefore, 
\begin{equation}
\chi((1+\theta) r) \leq \frac{2\Lambda }{b}( 1+\theta) ^{b+1}\bigg( \Big( 
\frac{\bar{\alpha}}{b^{\varepsilon }}\Big) ^{\frac{k}{2}}r^{b+1}+r^{b}\bigg)
\label{n3}
\end{equation}%
for all $r\geq r_{0}$ and $0<\theta \leq 1.$

By (\ref{n2}) and (\ref{n3}) we get 
\begin{align*}
\int_{r_{0}}^{r}&\chi ^{\frac{1}{q}}(( 1+\theta ) t) \left( \chi ^{\prime }(
t) \right) ^{1-\frac{1}{q}}t^{4m-2-\frac{1}{q}}dt \\
&\leq \frac{2\Lambda }{b^{\frac{1}{q}}}( 1+\theta) ^{\frac{b+1}{q}%
}\int_{r_{0}}^{r}\bigg( \Big( \frac{\bar{\alpha}}{b^{\varepsilon }}\Big) ^{%
\frac{k}{2}}t^{b}+t^{b-1}\bigg) t^{4m-2}dt \\
&\leq \frac{2\Lambda }{b^{1+\frac{1}{q}}}( 1+\theta) ^{\frac{b+1}{q}}\bigg( %
\Big( \frac{\bar{\alpha}}{b^{\varepsilon }}\Big) ^{\frac{k}{2}%
}r^{b+4m-1}+r^{b+4m-2}\bigg) .
\end{align*}%
Plugging into (\ref{id1}), we arrive at 
\begin{equation*}
\chi^{\prime \prime }( r) \leq \frac{2\Lambda C_{0}\bar{\alpha}}{\theta ^{%
\frac{2\nu }{q}}b^{1+\frac{1}{q}}}( 1+\theta) ^{\frac{b+1}{q}}\bigg( \Big( 
\frac{\bar{\alpha}}{b^{\varepsilon }}\Big) ^{\frac{k}{2}}r^{b-1}+r^{b-2}%
\bigg) +\frac{\Lambda _{0}}{r^{4m}}
\end{equation*}%
for all $r\geq r_{0}.$ Integrating in $r$ then yields 
\begin{equation}  \label{n4}
\chi^{\prime }(r) \leq \frac{3\Lambda C_{0}\bar{\alpha}}{\theta ^{\frac{2\nu 
}{q}}b^{2+\frac{1}{q}}}( 1+\theta) ^{\frac{b+1}{q}}\bigg(\Big( \frac{\bar{%
\alpha}}{b^{\varepsilon }}\Big) ^{\frac{k}{2}}r^{b}+r^{b-1}\bigg) +\frac{1}{2%
}\Lambda _{0}+\chi ^{\prime }\left( r_{0}\right)
\end{equation}%
for all $r\geq r_{0}$ and $0<\theta \leq 1.$ Note that by (\ref{Lambda}) 
\begin{equation*}
\chi ^{\prime }( r_{0}) =\frac{1}{r_{0}^{4m}}\int_{\Sigma \left(
r_{0}\right) }u\left\vert \nabla \rho \right\vert \leq \frac{1}{2}\Lambda
_{0}.
\end{equation*}%
Setting $\theta =\frac{1}{b}$ in (\ref{n4}) and using (\ref{eps}), we obtain
that 
\begin{equation*}
\chi ^{\prime }( r) \leq 4eC_{0}\frac{\bar{\alpha}}{b^{\varepsilon }}\Lambda %
\bigg( \Big( \frac{\bar{\alpha}}{b^{\varepsilon }}\Big) ^{\frac{k}{2}%
}r^{b}+r^{b-1}\bigg) +\Lambda _{0}.
\end{equation*}%
In view of (\ref{bC0}), 
\begin{equation*}
4eC_{0}\frac{\bar{\alpha}}{b^{\varepsilon }}\leq \frac{1}{2}\left( \frac{%
\bar{\alpha}}{b^{\varepsilon }}\right) ^{\frac{1}{2}}.
\end{equation*}%
Hence, the preceding inequality becomes 
\begin{equation*}
\chi^{\prime }( r) \leq \frac{1}{2}\Lambda \bigg( \Big( \frac{\bar{\alpha}}{%
b^{\varepsilon }}\Big) ^{\frac{k+1}{2}}r^{b}+r^{b-1}\bigg) +\Lambda _{0}.
\end{equation*}
However, 
\begin{equation*}
\Lambda _{0} \leq \Lambda \leq \frac{1}{2}\Lambda r^{b-1}
\end{equation*}%
for $r\geq r_{0}.$ In conclusion, 
\begin{equation*}
\chi ^{\prime }( r) \leq \Lambda \bigg( \Big( \frac{\bar{\alpha}}{%
b^{\varepsilon }}\Big) ^{\frac{k+1}{2}}r^{b}+r^{b-1}\bigg)
\end{equation*}%
for all $r\geq r_{0}.$

This completes the induction step and proves that (\ref{n2}) holds for all $k\geq 0.$ We
have thus established that 
\begin{equation}
\chi ^{\prime }( r) \leq \Lambda \bigg( \Big( \frac{\bar{\alpha}}{%
b^{\varepsilon }}\Big) ^{\frac{k}{2}}r^{b}+r^{b-1}\bigg)  \label{n9}
\end{equation}%
for all $k\geq 0$ and all $r\geq r_{0}.$

By (\ref{bC0}) we have $\frac{\bar{\alpha}}{b^{\varepsilon }}<1.$ Hence, by
letting $k\rightarrow \infty $ in (\ref{n9}) one sees that 
\begin{equation*}
\chi ^{\prime }( r) \leq \Lambda r^{b-1}
\end{equation*}%
for all $r\geq r_{0}.$ This clearly contradicts with the choice of $b.$

In conclusion, we must have 
\begin{equation}
\min \left\{ \frac{b^{\varepsilon }}{\bar{\alpha}},b\right\} \leq (
100C_{0}) ^{2}  \label{bC1}
\end{equation}%
for some constant $C_{0}=C_{0}( m,A_{0},\nu ,\widetilde{\alpha }). $

Let us consider first the case $q>\nu -\frac{1}{2}$ or $\varepsilon >0.$ It
is easy to see from (\ref{bC1}) that 
\begin{equation*}
b\leq \left( 100C_{0}\right) ^{\frac{2}{\varepsilon }}.
\end{equation*}
Therefore, 
\begin{equation*}
\int_{\Sigma \left( r\right) }\frac{u}{\left\vert \nabla \rho \right\vert }%
\leq \Lambda r^{\Gamma _{\epsilon }-1}
\end{equation*}%
for all $r\geq r_{0},$ where 
\begin{equation*}
\Gamma _{\varepsilon }=\left( 100C_{0}\right) ^{\frac{2}{\varepsilon }}+4m+1.
\end{equation*}
Integrating in $r$ and applying the mean value inequality (\ref{mean}), we
get 
\begin{equation}
u\leq \widetilde{\Lambda }\rho ^{\Gamma _{\epsilon }}\text{ \ on }%
M\backslash D( r_{0}) ,  \label{n10}
\end{equation}%
where $\widetilde{\Lambda }=\frac{2^{\Gamma _{\varepsilon }}\Lambda }{%
\mathrm{V}( R_{0}) }$.

Assume now that $q=\nu -\frac{1}{2}$ or $\varepsilon =0.$ Then (\ref{bC1})
implies 
\begin{equation}
\min \left\{ \frac{1}{\bar{\alpha}},b\right\} \leq \left( 100C_{0}\right)
^{2}.  \label{n11}
\end{equation}%
So if $\alpha< \alpha_0$ with 
\begin{equation*}
\frac{1}{\alpha_0}=\left( 100C_{0}\right) ^{2},
\end{equation*}%
then 
\begin{equation*}
\frac{1}{\bar{\alpha}}=\frac{1}{\alpha}>\left( 100C_{0}\right) ^{2}
\end{equation*}%
and (\ref{n11}) implies that 
\begin{equation*}
b\leq \left( 100C_{0}\right) ^{2}.
\end{equation*}%
As above, we conclude that 
\begin{equation}
u\leq \widetilde{\Lambda }\rho ^{\Gamma }\text{ \ on }M\backslash D( r_{0})
\label{n12}
\end{equation}%
for some $\Gamma( m,A_{0},\nu),$ where $\widetilde{\Lambda }=\frac{2^{\Gamma
}\Lambda }{\mathrm{V}( R_{0}) }.$

By (\ref{n10}) and (\ref{n12}), the theorem is proved.
\end{proof}

Combining Proposition \ref{G} with Theorem \ref{Growth}, we have the
following corollary concerning positive solutions $u$ to $\Delta u=\sigma u$
on $M\setminus D(R_{0}).$

\begin{corollary}
\label{BF} Assume that $\left( M,g\right)$ admits a proper function $\rho$
satisfying (\ref{L2}) and has the mean value property $\left( \mathcal{M}%
\right).$ Suppose that $\sigma$ decays quadratically. Then there exists $%
\Gamma( m, A_{0},\nu, \alpha)>0$ such that 
\begin{equation*}
u\leq \Lambda \left( \rho ^{\Gamma}+1\right) \text{ \ on }M\backslash D(
R_{0}),
\end{equation*}%
where $\Lambda>0$ is a constant depending on $u,$ provided that $%
\alpha<\infty$ for some $q>\nu -\frac{1}{2}.$ In the case $q=\nu -\frac{1}{2}%
,$ the same conclusion holds for some $\Gamma( m, A_{0},\nu)>0$ when $%
\alpha\leq \alpha _{0}( m, A_{0}, \nu),$ a sufficiently small positive
constant.
\end{corollary}

\section{Dimension Estimate\label{ED}}

In this section, we establish a dimension estimate for the space $\mathcal{P}
$ spanned by all positive solutions to the equation $\Delta u=\sigma u$ on $%
M.$ We continue to assume that $M$ admits a proper function $\rho$
satisfying (\ref{L2}) and has the mean value property $\left( \mathcal{M}%
\right).$ Our argument closely follows that in \cite{L}.

Define 
\begin{equation*}
L^d(M)=\left\{v: \Delta v=\sigma v, \left\vert v\right\vert \leq c\,\rho ^{d}%
\text{ \ on } M \right\},
\end{equation*}
the space of polynomial growth solutions of degree at most $d.$

\begin{lemma}
\label{dimpre1} Assume that $\left( M,g\right) $ admits a proper function $%
\rho $ satisfying (\ref{L2}) and has the mean value property $\left( 
\mathcal{M}\right) .$ Then $\dim L^{d}(M)\leq \Gamma (m,A_{0},\nu ,d).$
\end{lemma}

\begin{proof}
Let $\mathcal{W}_{l}$ be any $l$-dimensional subspace of $L^d(M),$ where $%
l>1.$ For $R>0,$ define the inner product 
\begin{equation*}
A_{R}( u,v) =\int_{D\left( R\right) }u\,v
\end{equation*}%
for $u, v\in \mathcal{W}_{l}.$ We claim that there exists $R>4R_{0}$ large
enough so that for $\left\{u_1,\cdots,u_l\right\}$, an orthonormal basis of $%
\mathcal{W}_{l}$ with respect to $A_{2R}$, 
\begin{equation}
\sum_{i=1}^{l}\int_{D\left( R\right) }u_{i}^{2}\geq \frac{l}{\bar{\Gamma}},
\label{z2}
\end{equation}%
where $\bar{\Gamma}=2^{\gamma ( m) +2\,d+1}$ with $\gamma(m)$ being the same
constant from Lemma \ref{Area}.

Indeed, assume by contradiction that (\ref{z2}) fails for all $R>4R_{0}.$ To
simplify notation, for $R_{2}>R_{1},$ we denote by 
\begin{equation*}
\mathrm{tr}_{A_{R_{2}}}A_{R_{1}}=\sum_{i=1}^{l}\int_{D\left( R_{1}\right)
}v_{i}^{2}
\end{equation*}%
for orthonormal basis $\left\{v_1,\cdots,v_{l}\right\}$ with respect to $%
A_{R_{2}}.$ Since (\ref{z2}) fails for all $R>4R_{0},$ we have that 
\begin{equation*}
\frac{1}{\bar{\Gamma}} >\frac{\mathrm{tr}_{A_{2R}}A_{R}}{l} \geq \left( 
\mathrm{det}_{A_{2R}}A_{R}\right) ^{\frac{1}{l}},
\end{equation*}%
where the last estimate follows from the arithmetic-geometric mean
inequality. In other words, 
\begin{equation}
\mathrm{det}_{A_{2R}}A_{R}\leq \frac{1}{\bar{\Gamma}^{l}}  \label{z3}
\end{equation}%
for all $R\geq 4R_{0}.$ Iterating (\ref{z3}) and using that 
\begin{equation*}
\left( \mathrm{det}_{A_{T}}A_{R}\right) \left( \mathrm{det}%
_{A_{R}}A_{S}\right) =\mathrm{det}_{A_{T}}A_{S},
\end{equation*}%
we get 
\begin{equation*}
\mathrm{det}_{A_{2^{j}R}}A_{R}\leq \frac{1}{\bar{\Gamma}^{lj}}.
\end{equation*}%
Equivalently, 
\begin{equation}
\mathrm{det}_{A_{R}}A_{2^{j}R}\geq \bar{\Gamma}^{lj}  \label{z4}
\end{equation}%
for all $j>0$ and $R\geq 4R_{0}.$

On the other hand, Lemma \ref{Area} implies that $\mathrm{V}( 2^{j}R) \leq
(2^{j}R) ^{\gamma( m) }\mathrm{V}(R_0).$ Together with the fact that $u\in \mathcal{W}_{l}$
is of polynomial growth of order at most $d$, we conclude 
\begin{equation*}
\mathrm{det}_{A_{R}}A_{2^{j}R}\leq \Lambda ^{2l}( 2^{j}R) ^{( \gamma \left(
m\right) +2\,d ) l}\mathrm{V}(R_0)^l.
\end{equation*}
As $\bar{\Gamma}>2^{\gamma \left( m\right) +2\,d},$ this contradicts (\ref%
{z4}) after letting $j\rightarrow \infty.$ This proves (\ref{z2}).

For $x\in \Sigma(R) $ we note that there exists a subspace $\mathcal{W}_{x}$ of $%
\mathcal{W}_{l},$ of codimension at most one, such that $u( x) =0$ for all $%
u\in \mathcal{W}_{x}.$ So one may choose an orthonormal basis in $\mathcal{W}%
_{l}$ with $u_{2},\cdots,u_{l}\in \mathcal{W}_{x}.$ By the mean value
property $\left( \mathcal{M}\right)$ we get 
\begin{align*}
\sum_{i=1}^{l}u_{i}^{2}( x) &=u_{1}^{2}( x) \\
&\leq \frac{C( A_0,\mu) }{\mathrm{V}( 2R) }\int_{D\left( 2R\right) }u_{1}^{2}
\\
&=\frac{C( A_0,\mu ) }{\mathrm{V}( 2R) }.
\end{align*}
The function $\Psi(x)=\sum_{i=1}^{l}u_{i}^{2}( x)$ is subharmonic, therefore its maximum 
on $D(R)$ is achieved on $\Sigma(R)$. 
We have thus proved that 
\begin{equation*}
\sum_{i=1}^{l}u_{i}^{2}( x) \leq \frac{C( A_0,\mu) }{\mathrm{V}( 2R) }
\end{equation*}%
for $x\in D( R).$ Together with (\ref{z2}) we get 
\begin{align*}
\frac{l}{\bar{\Gamma}} &\leq \sum_{i=1}^{l}\int_{D\left( R\right)
}u_{i}^{2}( x) \\
&\leq \frac{C( A_0,\mu) }{\mathrm{V}( 2R) }\mathrm{V}( R).
\end{align*}%
Therefore, 
\begin{equation*}
l\leq C( A_0,\mu) \bar{\Gamma}.
\end{equation*}%
Since this holds true for any $l$-dimensional subspace $\mathcal{W}_{l}$ of $%
L^d(M),$ we conclude that 
\begin{equation*}
\dim L^d(M)\leq C( A_0,\mu) \bar{\Gamma}
\end{equation*}
as well. This proves the result.
\end{proof}

Summarizing, we have the following theorem. Recall $\mathcal{P}$ is the
space spanned by all positive solutions to the equation $\Delta u=\sigma u.$

\begin{theorem}
\label{dimpre} Assume that $\left( M,g\right)$ admits a proper function $%
\rho $ satisfying (\ref{L2}) and has the mean value property $\left( 
\mathcal{M}\right).$ Suppose that $\sigma$ decays quadratically. Then $\dim 
\mathcal{P}\leq \Gamma( m, A_{0},\nu, \alpha)$ provided that $\alpha<\infty$
for some $q>\nu -\frac{1}{2}.$ In the case $q=\nu -\frac{1}{2},$ the same
conclusion holds for some $\Gamma ( m, A_{0},\nu)$ when $\alpha\leq \alpha
_{0}( m, A_{0}, \nu ),$ a sufficiently small positive constant.
Consequently, the number of ends $e(M)$ of $M$ satisfies the same estimate
as well.
\end{theorem}

\begin{proof}
According to Theorem \ref{E}, the number of ends $e(M)$ is at most the
dimension of $\mathcal{P}.$ However, Corollary \ref{BF} implies that $%
\mathcal{P}\subset L^{d}(M)$ with $d=\Gamma (m,A_{0},\nu ,\alpha )$ in the
case $q>\nu -\frac{1}{2}$ and $d=\Gamma (m,A_{0},\nu )$ in the case $q=\nu -%
\frac{1}{2},$ respectively. The conclusion on the dimension estimate of $%
\mathcal{P}$ then follows from Lemma \ref{dimpre1}. This proves the theorem.
\end{proof}

\section{Sobolev inequality\label{Sob}}

In this section, we show that a scaling invariant Sobolev inequality implies
the mean value property $( \mathcal{M})$, a classical fact proven by a
well-known Moser iteration argument. For the sake of completeness, we will
spell out the details below. We continue to assume that $M$ admits a proper
Lipschitz function $\rho >0$ satisfying (\ref{L}), namely, 
\begin{equation}
\frac{1}{2}\leq \left\vert \nabla \rho \right\vert \leq 1\text{\ and\ }
\Delta \rho \leq \frac{m}{\rho }  \label{L3}
\end{equation}%
in the weak sense for $\rho \geq R_{0}.$ The sublevel and level sets of $%
\rho $ are denoted by 
\begin{align*}
D\left( r\right) &=\left\{ x\in M:\rho( x) <r\right\} \\
\Sigma \left( r\right) &=\left\{ x\in M:\rho( x) =r\right\},
\end{align*}%
respectively, and their volume and area by 
\begin{align*}
\mathrm{V}( r) &=\mathrm{Vol}( D( r)) \\
\mathrm{A}( r) &=\mathrm{Area}( \Sigma ( r)).
\end{align*}

Recall that $\left( M,g\right) $ satisfies the Sobolev inequality $(\mathcal{%
S})$ if there exist constants $\mu>1$ and $A>0$ such that 
\begin{equation}
\left( \fint_{D\left( R\right) }\phi ^{2\mu }\right) ^{\frac{1}{\mu }}\leq
AR^{2}\fint_{D\left( R\right) }\left( \left\vert \nabla \phi \right\vert
^{2}+\sigma \phi ^{2}\right)  \label{Sob2}
\end{equation}%
for $\phi \in C_{0}^{\infty }( D( R))$ and $R\geq R_{0}.$ Here and in the
following, 
\begin{equation*}
\fint_{\Omega }u=\frac{1}{\mathrm{Vol}( \Omega) }\int_{\Omega }u
\end{equation*}%
for a compact subset $\Omega \subset M$ and an integrable function $u$ on $%
\Omega$. We denote $\nu$ to be the number determined by 
\begin{equation*}
\frac{1}{\mu}+\frac{1}{\nu}=1.
\end{equation*}

\begin{proposition}
\label{MVI} Assume that $\left( M,g\right)$ admits a proper function $\rho$
satisfying (\ref{L3}) and that the Sobolev inequality $(\mathcal{S})$ holds.
Then there exists a constant $C( A,\mu)>0$ such that 
\begin{equation*}
\sup_{\Sigma \left( R\right) }u\leq \frac{C( A,\mu) }{\theta ^{2\nu }\mathrm{%
V}( 2R) }\int_{D\left( \left( 1+\theta \right) R\right) \backslash D\left( 
\frac{R}{4}\right) }u
\end{equation*}%
for any $0<\theta \leq 1$ and a positive subsolution $u$ of $\Delta u\geq
\sigma u$ on $D( 2R) \backslash D( R_{0})$ with $R\geq 4R_{0}.$ In
particular, $M$ has the mean value property $( \mathcal{M})$.
\end{proposition}

\begin{proof}
The proof is by Moser iteration and can be found in Chapter 19 of \cite{Li}.
We may assume $0<\theta <\frac{1}{8}.$ For a function $\phi $ with compact
support in $D( 2R) $ and a positive integer $k\geq 1,$ applying the Sobolev
inequality (\ref{Sob2}) to $\phi u^{k},$ we get 
\begin{equation}
\left( \int_{D\left( 2R\right) }\left( u^{k}\phi \right) ^{2\mu }\right) ^{%
\frac{1}{\mu }}\leq \frac{4AR^{2}}{\mathrm{V}( 2R) ^{\frac{1}{\nu }}}%
\int_{D\left( 2R\right) }\left( \left\vert \nabla \left( u^{k}\phi \right)
\right\vert ^{2}+\sigma u^{2k}\phi ^{2}\right),  \label{m12}
\end{equation}%
where $\frac{1}{\nu }=1-\frac{1}{\mu }.$ Integrating by parts and using $%
\Delta u\geq \sigma u,$ we compute the first term of the right side as 
\begin{align*}
\int_{D\left( 2R\right) }\left\vert \nabla \left( u^{k}\phi \right)
\right\vert ^{2} &=k^{2}\int_{D\left( 2R\right) }\left\vert \nabla
u\right\vert ^{2}u^{2k-2}\phi ^{2}+\int_{D\left( 2R\right) }\left\vert
\nabla \phi \right\vert ^{2}u^{2k} \\
&\quad +\frac{1}{2}\int_{D\left( 2R\right) }\left\langle \nabla
u^{2k},\nabla \phi ^{2}\right\rangle \\
&=-k\left( k-1\right) \int_{D\left( 2R\right) }\left\vert \nabla
u\right\vert ^{2}u^{2k-2}\phi ^{2}-k\int_{D\left( 2R\right) }\left( \Delta
u\right) u^{2k-1}\phi ^{2} \\
&\quad +\int_{D\left( 2R\right) }\left\vert \nabla \phi \right\vert
^{2}u^{2k} \\
&\leq -\int_{D\left( 2R\right) }\sigma u^{2k}\phi ^{2}+\int_{D\left(
2R\right) }\left\vert \nabla \phi \right\vert ^{2}u^{2k}.
\end{align*}%
Plugging into (\ref{m12}) we conclude 
\begin{equation}
\bigg( \int_{D\left( 2R\right) }\left( u^{k}\phi \right) ^{2\mu }\bigg) ^{%
\frac{1}{\mu }}\leq \frac{4AR^{2}}{\mathrm{V}( 2R) ^{\frac{1}{\nu }}}%
\int_{D\left( 2R\right) }u^{2k}\left\vert \nabla \phi \right\vert ^{2}.
\label{m13}
\end{equation}%
For fixed constants $T_1,$ $T_2,$ $\delta_1$ and $\delta_2$ with $\frac{R}{2}%
<T_{1}<T_{2}<\frac{3R}{2}$ and $0<\delta_{1},\ \delta _{2}<\frac{1}{4}R,$
let 
\begin{equation*}
\phi ( x) =\left\{ 
\begin{array}{llll}
1 & \text{on }D( T_{2}) \backslash D( T_{1}) &  &  \\ 
\frac{1}{\delta _{2}}( T_{2}+\delta _{2}-\rho( x)) & \text{on }D(
T_{2}+\delta _{2}) \backslash D( T_{2}) &  &  \\ 
\frac{1}{\delta _{1}}( \rho( x) -T_{1}+\delta _{1}) & \text{on }D( T_{1})
\backslash D( T_{1}-\delta _{1}) &  &  \\ 
0 & \text{otherwise.} &  & 
\end{array}
\right.
\end{equation*}%
Plugging into (\ref{m13}) we get 
\begin{equation}
\left\Vert u\right\Vert _{2k\mu ,T_{1},T_{2}}\leq \left( \frac{4AR^{2}}{%
\mathrm{V}\left( 2R\right) ^{\frac{1}{\nu }}\min \{ \delta _{1},\delta
_{2}\} ^{2}}\right) ^{\frac{1}{2k}}\left\Vert u\right\Vert _{2k,T_{1}-\delta
_{1},T_{2}+\delta _{2}},  \label{m14}
\end{equation}%
where 
\begin{equation*}
\left\Vert u\right\Vert _{a,T_{1},T_{2}}=\left( \int_{D\left( T_{2}\right)
\backslash D\left( T_{1}\right) }u^{a}\right) ^{\frac{1}{a}}.
\end{equation*}%
We now iterate the inequality. Fix $\frac{3R}{8}<R_{1}<R_{2}<\frac{5}{4}R$
and $0<\epsilon _{1},\ \epsilon _{2}<\frac{1}{8}.$ For each integer $i\geq
0, $ set 
\begin{align*}
k_{i} &=\mu ^{i} \\
\delta _{1,i} &=\frac{\epsilon _{1}R_{1}}{2^{i+1}},\ \ \ \delta _{2,i}=\frac{%
\epsilon _{2}R_{2}}{2^{i+1}} \\
T_{1,i} &=\left( 1-\epsilon _{1}\right) R_{1}+\sum_{j=0}^{i}\delta _{1,j},\
\ \ T_{2,i}=\left( 1+\epsilon _{2}\right) R_{2}-\sum_{j=0}^{i}\delta _{2,j}.
\end{align*}%
Applying (\ref{m14}) with $k=k_{j},$ $\delta _{1}=\delta _{1,j},$ $\delta
_{2}=\delta _{2,j}$ and $T_{1}=T_{1,j}$ and $T_{2}=T_{2,j},$ and iterating
from $j=0$ to $j=i,$ one obtains 
\begin{equation*}
\left\Vert u\right\Vert _{2\mu ^{i+1},T_{1,i},T_{2,i}}\leq
\prod\limits_{j=0}^{i}\left( \frac{4AR^{2}}{\mathrm{V}( 2R) ^{\frac{1}{\nu }%
}\min\{ \delta _{1,j},\delta _{2,j}\} ^{2}}\right) ^{\frac{1}{2\mu ^{j}}%
}\left\Vert u\right\Vert _{2,\left( 1-\epsilon _{1}\right) R_{1},\left(
1+\epsilon _{2}\right) R_{2}}.
\end{equation*}%
Letting $i\rightarrow \infty$ yields 
\begin{equation*}
\left\Vert u\right\Vert _{\infty ,R_{1},R_{2}}\leq \left( \frac{C( \mu ) A}{%
\mathrm{V}( 2R) ^{\frac{1}{\nu }}\min\{ \epsilon _{1},\epsilon _{2}\} ^{2}}%
\right) ^{\frac{\nu }{2}}\left\Vert u\right\Vert _{2,\left( 1-\epsilon
_{1}\right) R_{1},\left( 1+\epsilon _{2}\right) R_{2}}
\end{equation*}%
for $\frac{3R}{8}<R_{1}<R_{2}<\frac{5}{4}R$ and $0<\epsilon_{1},\ \epsilon
_{2}<\frac{1}{8}.$

So we have 
\begin{equation}  \label{m15}
\begin{split}
&\left\Vert u\right\Vert _{\infty ,R_{1},R_{2}} \leq \frac{C\left( A,\mu
\right) }{\mathrm{V}\left( 2R\right) ^{\frac{1}{2}}\min \left\{ \epsilon
_{1},\epsilon _{2}\right\} ^{\nu }}\left\Vert u\right\Vert _{2,\left(
1-\epsilon _{1}\right) R_{1},\left( 1+\epsilon _{2}\right) R_{2}} \\
&\leq \frac{C\left( A,\mu \right) }{\mathrm{V}\left( 2R\right) ^{\frac{1}{2}%
}\min \left\{ \epsilon _{1},\epsilon _{2}\right\} ^{\nu }}\left\Vert
u\right\Vert _{\infty ,\left( 1-\epsilon _{1}\right) R_{1},\left( 1+\epsilon
_{2}\right) R_{2}}^{\frac{1}{2}}\left\Vert u\right\Vert _{1,\left(
1-\epsilon _{1}\right) R_{1},\left( 1+\epsilon _{2}\right) R_{2}}^{\frac{1}{2%
}}.
\end{split}%
\end{equation}
Applying (\ref{m15}) for each $i$ with 
\begin{align*}
R_1=R_{1,i} &=\frac{R}{2}-\frac{\theta R}{2}\sum_{j=1}^{i}\frac{1}{2^{j}},\
\ \ \ \ \epsilon_1=\epsilon _{1,i}=1-\frac{R_{1,i+1}}{R_{1,i}} \\
R_2=R_{2,i} &=R+\theta R\sum_{j=1}^{i}\frac{1}{2^{j}},\ \ \ \ \ \ \,
\epsilon_2=\epsilon _{2,i}=\frac{R_{2,i+1}}{R_{2,i}}-1
\end{align*}%
and iterating, we conclude that 
\begin{equation*}
\left\Vert u\right\Vert _{\infty ,\frac{R}{2},R}\leq \frac{C( A,\mu ) }{%
\mathrm{V}( 2R) \theta ^{2\nu }}\left\Vert u\right\Vert _{1,( 1-\theta ) 
\frac{R}{2},( 1+\theta ) R}.
\end{equation*}%
This proves the result.
\end{proof}

We note that only $|\nabla \rho|\leq 1$ on $M\setminus D(R_0)$ from (\ref{L3}%
) was used in the proof of Proposition \ref{MVI}. The following corollary is
immediate.

\begin{corollary}
\label{MVIC} Assume that $\left( M,g\right)$ admits a proper function $\rho$
satisfying (\ref{L3}) and that the Sobolev inequality $(\mathcal{S})$ holds.
Then there exists $C( A,\mu)>0$ such that 
\begin{equation*}
\sup_{D\left( R\right) }u\leq \frac{C\left( A,\mu \right) }{\theta ^{2\nu }}%
\fint_{D\left( \left( 1+\theta \right) R\right) }u
\end{equation*}%
for any $0<\theta \leq 1$ and positive subsolution $u$ of $\Delta u\geq
\sigma u$ on $D( 2R)$ with $R\geq R_{0}.$
\end{corollary}

By combining Proposition \ref{MVI} with Theorem \ref{dimpre}, we have the
following result.

\begin{theorem}
\label{dim} Assume that $\left( M,g\right)$ admits a proper function $\rho$
satisfying (\ref{L3}) and that the Sobolev inequality $(\mathcal{S})$ holds.
Suppose that $\sigma$ decays quadratically. Then $\dim \mathcal{P}\leq
\Gamma ( m, A,\nu, \alpha )$ provided that $\alpha<\infty$ for some $q>\nu -%
\frac{1}{2}.$ In the case $q=\nu -\frac{1}{2},$ the same conclusion holds
for some $\Gamma( m, A,\nu)$ when $\alpha\leq \alpha _{0}( m, A, \nu),$ a
sufficiently small positive constant. Consequently, the number of ends $e(M)$
of $M$ satisfies the same estimate as well.
\end{theorem}

We also remark that Proposition \ref{MVI} can be localized to an end $E$ of $%
M$ as follows. For $\ r\geq R_{0},$ we denote 
\begin{align*}
E( r) &=E\cap D( r), \\
\partial E( r) &=E\cap \Sigma( r).
\end{align*}

\begin{corollary}
\label{MVILoc} Assume that $\left( M,g\right)$ admits a proper function $%
\rho $ satisfying (\ref{L3}) and that the Sobolev inequality $(\mathcal{S})$
holds. Then there exists a constant $C( A,\mu)>0$ such that 
\begin{equation*}
\sup_{\partial E\left( R\right) }u\leq \frac{C( A,\mu) }{\theta ^{2\nu }%
\mathrm{V}( 2R) }\int_{E\left( \left( 1+\theta \right) R\right) \backslash
E\left( \frac{R}{4}\right) }u
\end{equation*}%
for any $0<\theta \leq 1$ and positive subsolution $u$ of $\Delta u\geq
\sigma u$ on $E( 2R) \backslash E( R_{0})$ with $R\geq 4R_{0}.$
\end{corollary}

\begin{proof}
In the proof of Proposition \ref{MVI} one may choose the cut-off $\phi$ with
support in the end $E$ as follows. 
\begin{equation*}
\phi( x) =\left\{ 
\begin{array}{llll}
1 & \text{on }E( T_{2}) \backslash D( T_{1}) &  &  \\ 
\frac{1}{\delta _{2}}( T_{2}+\delta _{2}-\rho( x)) & \text{on }E(
T_{2}+\delta _{2}) \backslash D( T_{2}) &  &  \\ 
\frac{1}{\delta _{1}}( \rho( x) -T_{1}+\delta _{1}) & \text{on }E( T_{1})
\backslash D( T_{1}-\delta _{1}) &  &  \\ 
0 & \text{otherwise.} &  & 
\end{array}
\right.
\end{equation*}%
with $\frac{R}{2}<T_{1}<T_{2}<\frac{3R}{2}$ and $0<\delta_{1},\ \delta _{2}<%
\frac{1}{4}R.$ The rest of the proof is verbatim.
\end{proof}

It is perhaps worth pointing out that the normalization in Corollary \ref%
{MVILoc} is by the volume of $D(2R),$ not of its intersection with $E.$ We
now apply this localized version to improve Corollary \ref{BF}.

For an end $E$ of $M,$ define 
\begin{equation}
\alpha _{E}=\limsup_{R\rightarrow \infty }\left( \frac{R^{2q}}{\mathrm{A}(
R) }\int_{\partial E\left( R\right) }\sigma ^{q}\right) ^{\frac{1}{q}}.
\label{ae}
\end{equation}

\begin{corollary}
\label{BFLoc} Assume that $\left( M,g\right)$ admits a proper function $\rho$
satisfying (\ref{L3}) and that the Sobolev inequality $(\mathcal{S})$ holds.
Suppose that $\sigma$ decays quadratically along $E.$ Then there exists $%
\Gamma( m, A, \nu, \alpha_E)>0$ such that 
\begin{equation*}
u\leq \Lambda \left( \rho ^{\Gamma}+1\right) \text{ \ on } E
\end{equation*}%
for any positive solution $u$ to $\Delta u=\sigma u$ on $E$, where $%
\Lambda>0 $ is a constant depending on $u$, provided that $\alpha_E<\infty$
for some $q>\nu -\frac{1}{2}.$ In the case $q=\nu -\frac{1}{2},$ the same
conclusion holds for some $\Gamma( m, A,\nu)>0$ when $\alpha_E\leq \alpha
_{0}( m, A, \nu)$, a sufficiently small positive constant.
\end{corollary}

\begin{proof}
First, Lemma \ref{I} can be localized to the end $E$ to yield 
\begin{equation*}
\frac{d}{dr}\left( \frac{1}{r^{4m}}\int_{\partial E\left( r\right)
}u\left\vert \nabla \rho \right\vert \right) \leq \frac{1}{r^{4m}}%
\int_{E\left( r\right) \backslash E\left( r_{0}\right) }\sigma u+\frac{1%
}{r^{4m}}\int_{\partial E\left( r_{0}\right) }\frac{\left\langle \nabla
u,\nabla \rho \right\rangle }{\left\vert \nabla \rho \right\vert }
\end{equation*}%
for any $r_{0}\geq R_{0}.$ Using the fact that $\sigma $ decays
quadratically along $E$, one concludes that $u$ is of polynomial growth
along $E$ by adopting the same argument as in Proposition \ref{G}.

Recall by Corollary \ref{MVILoc} that 
\begin{equation}
\sup_{\partial E\left( R\right) }u\leq \frac{C( A,\mu) }{\theta ^{2\nu }}%
\frac{1}{\mathrm{V}( 2R) }\int_{E\left( \left( 1+\theta \right) R\right)
\backslash E\left( R_{0}\right) }u  \label{Mloc}
\end{equation}%
for $R>4R_{0}$ and $0<\theta \leq 1.$ Following the proof of (\ref{id}) we
obtain that the function 
\begin{equation*}
\chi _{E}( r) =\int_{E\left( r\right) \backslash E\left( R_{0}\right) }u
\frac{\left\vert \nabla \rho \right\vert ^{2}}{\rho ^{4m}}
\end{equation*}%
satisfies the following inequality: 
\begin{equation*}
r^{4m}\chi _{E}^{\prime \prime }( r) \leq \frac{C_{0}\bar{\alpha}_{E}}{%
\theta ^{\frac{2\nu }{q}}}\int_{r_{0}}^{r}\chi _{E}(( 1+\theta) t)^{\frac{1}{%
q}}( \chi _{E}^{\prime }( t))^{1-\frac{1}{q}}t^{4m-2-\frac{1}{q}}dt+\Lambda
_{0}
\end{equation*}%
for $r\geq r_{0}$ and $0<\theta \leq 1,$ where 
\begin{equation*}
\Lambda _{0}=\int_{\partial E\left( r_{0}\right) }\left( u+\left\vert
\nabla u\right\vert \right)
\end{equation*}%
and $\bar{\alpha}_{E}=\min \left\{ \alpha _{E},1\right\},$ with the constant 
$C_{0}$ depending only on $m, A, \mu $ and $\alpha _{E}.$

Using an induction argument as in Theorem \ref{Growth}, we arrive at 
\begin{equation*}
\int_{\partial E\left( r\right) }u\leq \Lambda r^{C( m,A,\mu, \alpha
_{E}) }
\end{equation*}%
for $r\geq r_{0}.$ Integrating in $r$ and using (\ref{Mloc}), we conclude 
\begin{equation*}
u\leq \widetilde{\Lambda }\left( \rho ^{\Gamma _{\epsilon}}+1\right)
\end{equation*}
on end $E.$ This proves the result.
\end{proof}

Corresponding to an end $E,$ let $u_{E}$ be the positive solution of $\Delta
u_E=\sigma u_{E}$ on $M$ constructed in Theorem \ref{E}. Then $0<u_{E}\leq 1$
on $M\setminus E$. In particular, under the assumptions of Corollary \ref%
{BFLoc}, $u_E$ must be of polynomial growth on $M$ with the given growth
order.

\vspace{0.2 in}

\textbf{Acknowledgment}. The first author was partially supported by NSF
grant DMS-1506220 and by a Leverhulme Trust Visiting Professorship
VP2-2018-029. The second author was supported by a Leverhulme Trust Research
Project Grant RPG-2016-174.

\end{document}